\documentclass[aos, preprint]{imsart}

\usepackage{amsthm, amsmath, amssymb, amscd, latexsym}
\usepackage{color}
\usepackage{geometry}
\startlocaldefs

\newtheorem {theorem}{Theorem}[section]
\newtheorem {lemma}[theorem]{Lemma}
\newtheorem {proposition}[theorem]{Proposition}
\newtheorem {definition}[theorem]{Definition}

\newtheorem {corollary}[theorem]{Corollary}

\newtheorem {remark}[theorem]{Remark}
\newtheorem {notation}[theorem]{Notation}

\def \N {{\mathbb N}}
\def \R {{\mathbb R}}

\newcommand{\norm}[1]{\left\|#1 \right\|}

\newcommand{\T}{\mathbb{T}}

\newcommand{\cI}{\mathcal{I}}

\newcommand{\bbE}{\mathbb{E}}
\def\expec{\mathbb{E}}

\def \sign{{\rm sign}}

\newcommand{\gep}{\epsilon}
\newcommand{\gl}{\lambda}
\newcommand{\ga}{\alpha}
\newcommand{\gs}{r}
\newcommand{\gd}{\delta}
\newcommand{\G}{\mathcal{G}}
\newcommand{\la}{\langle}
\newcommand{\ra}{\rangle}
\newcommand{\pr}{\mu_0}
\newcommand{\post}{\mu^y}
\newcommand{\dpost}{\pi^y}

\newcommand{\e}{{\rm e}}
\newcommand{\ud}{d} 
\newcommand{\tr}{\tilde R}
\newcommand{\utr}{u^{\dagger}}

\newcommand{\ts}{{\tilde s}}
\newcommand{\diam}{{\epsilon}}

\definecolor{darkred}{rgb}{.7,0,0}

\definecolor{darkgreen}{rgb}{0,0.5,0}

\definecolor{darkblue}{rgb}{0,0,0.7}

\endlocaldefs

\begin{document}

\begin{frontmatter}

\title{Sparsity-promoting and edge-preserving maximum a posteriori estimators in non-parametric Bayesian inverse problems}

\runtitle{MAP estimators and Besov prior}



\begin{aug}
\author{\fnms{Sergios} \snm{Agapiou}\thanksref{m1}\ead[label=e1]{agapiou.sergios@ucy.ac.cy}}
\author{\fnms{Martin} \snm{Burger}\thanksref{m2}\ead[label=e2]{martin.burger@wwu.de}}
\author{\fnms{Masoumeh} \snm{Dashti}\thanksref{m3}\ead[label=e3]{m.dashti@sussex.ac.uk}}
\and\author{\fnms{Tapio} \snm{Helin}\thanksref{m4}\ead[label=e4]{tapio.helin@helsinki.fi}}

\runauthor{S. Agapiou, M. Burger, M. Dashti and T. Helin}

\affiliation{University of Cyprus \thanksmark{m1},  University of M\"unster\thanksmark{m2}, University of Sussex \thanksmark{m3} and University of Helsinki\thanksmark{m4}}

\address{S. Agapiou\\ Department of Mathematics and Statistics\\
1 University Avenue\\ 2109 Nicosia\\ Cyprus\\
\printead{e1}}

\address{M. Burger\\
Institute for Computational\\ and Applied Mathematics\\
Westf\"alische Wilhelms-Universit\"at M\"unster\\
and Cells in Motion Cluster of Excellence\\
M\"unster\\
Germany\\
\printead{e2}}

\address{M. Dashti\\
Department of Mathematics\\
University of Sussex\\
BN1 5DJ Brighton\\
United Kingdom\\
\printead{e3}}

\address{T. Helin\\
Department of Mathematics and Statistics\\
University of Helsinki\\
Gustaf Hallstr\"omin katu 2b\\
FI-00014 Helsinki\\
Finland\\
\printead{e4}}
\end{aug}

\begin{abstract}
We consider the inverse problem of recovering an unknown functional parameter $u$ in a separable Banach space, from a noisy observation $y$ of its image through a known possibly non-linear ill-posed map ${\mathcal G}$. The data $y$ is finite-dimensional and the noise is Gaussian. We adopt a Bayesian approach to the problem and
consider Besov space priors (see \cite{LSS09}), which are well-known for their edge-preserving and sparsity-promoting properties and have recently attracted wide attention especially in the medical imaging community. 

Our key result is to show that in this non-parametric setup the maximum a posteriori (MAP) estimates are characterized by the minimizers of a generalized Onsager--Machlup functional of the posterior.
This is done independently for the so-called weak and strong MAP estimates, which as we show coincide in our context. 
In addition, we prove a form of weak consistency for the MAP estimators in the infinitely informative data limit.
Our results are remarkable for two reasons: first, the prior 
distribution is non-Gaussian and does not meet the smoothness conditions required in previous research on non-parametric MAP estimates.
Second, the result analytically justifies existing uses of the MAP estimate in finite but high dimensional discretizations of Bayesian inverse problems with the considered Besov priors.

\end{abstract}



\end{frontmatter}

\section{Introduction}

We consider the inverse problem of recovering an unknown functional parameter $u\in X$ from a noisy and indirect observation $y\in Y.$ We work in a framework in which $X$ is an infinite-dimensional separable Banach space, while $Y=\R^J$. 
In particular, we consider the additive noise model
\begin{equation}
\label{eq:maineq}
y = {\mathcal G}(u) +  \xi,
\end{equation} where $\xi$ is mean zero Gaussian observational noise, with a positive definite covariance matrix $\Sigma\in\R^{J\times J}$, $\xi\sim N(0,\Sigma).$
Here the possibly nonlinear operator $\G:X\to Y$ describes the system response, connecting the observation $y$ to the unknown parameter $u$. More specifically, $\G$ captures both the forward model and the observation mechanism and is assumed to be known.

Inverse problems are mathematically characterized by being ill-posed: the lack of sufficient information in the observation prohibits the unique or stable reconstruction of the unknown. This can either be due to an inherent loss of information in the forward model, or due to the incomplete and noisy observation. 
To address ill-posedness, regularization techniques are employed in which the available information in the observation is augmented using \emph{a priori} available knowledge on the properties of candidate solutions. 

We adopt a Bayesian approach to the regularization of inverse problems, which
has in recent years attracted enormous attention in the inverse problems and imaging literature: see the early work \cite{JF70}, the books \cite{AT05, KS06} and the more recent works \cite{LSS09, Stuart, DS15} and references therein.
In this approach, prior information is encoded in the \emph{prior} distribution $\pr$ on the unknown $u$, and the Bayesian methodology is used to formally obtain the \emph{posterior} distribution $\post$ on $u|y$, 
in the form 
\begin{equation}\label{eq:intpost}
\post(du|y)\propto\exp\big(-\Phi(u;y)\big)\pr(du),
\end{equation} 
where \[\Phi(u;y)=\frac12\norm{\Sigma^{-\frac12}\big(y-\G(u)\big)}^2_Y.\] 

Our work is driven by two types of a priori information: on the one hand, we aim to recover unknown functions with a blocky structure, as is typically the case in image processing \cite{DS96} and medical imaging applications \cite{DS94}.
On the other hand, we are interested in prior models which promote sparse solutions, that is high-dimensional solutions that can be represented by a small number of coefficients in an appropriate expansion.
To achieve these effects, we utilize so-called Besov space priors \cite{LSS09, Dashti_besov}, which are well-known for their edge-preserving and sparsity-promoting properties, see e.g. \cite{KLNS12, rantala2006wavelet,leporini2001bayesian,jia2016bayesian,bui2015scalable}. 
With these goals in mind, our work relates to classical $L^1$-type regularization methods such as penalized least squares with Total Variation penalty \cite{ROF92, SC03}.

In practice, the implementation of the Bayesian approach to an inverse problem typically requires a high-dimensional discretization of the unknown and large computational resources.
From a computational perspective, it is hence imperative that the probability models and estimators related to finite-dimensional discretizations of the problem scale well with respect to refining discretization.
In particular, there is a fundamental need to understand whether specialized prior information used frequently in applications, for example of the type described above, leads to well-defined non-parametric probability models and whether the related finite-dimensional estimators have well-behaving limits; this motivates the study of Bayesian inverse problems in the infinite-dimensional function-space setting.
A body of literature on this topic has emerged during the last years, centered around two theories:
\begin{enumerate}
\item[a)] the \emph{discretization invariance} theory \cite{LS04, LSS09}, which aims to ensure that when investing more resources in increasing the discretization level, the prior remains faithful to the intended information on the unknown function and the posterior converges to a well defined limit which is an improved representation of the reality;
\item[b)] the \emph{well-posedness} of the posterior theory found in \cite{Stuart, DS15}, which secures that the posterior is well defined in the infinite-dimensional limit and robust with respect to perturbations in the data as well as approximations of the forward model. 
\end{enumerate}

Our work studies Bayesian inversion in function spaces especially from the perspective of point estimators. Namely, we study and give a rigorous meaning to  \emph{maximum a posteriori} (MAP) estimates for Bayesian inverse problems with certain Besov priors in the infinite-dimensional function-space setting. A MAP estimate is understood here as the mode of the posterior probability measure, hence our results require a careful definition of a mode in the infinite-dimensional setting. The main challenge is then to establish a connection between the topological definition of a MAP estimate (mode of the posterior) and an explicit variational problem. Succeeding in doing so, opens up the possibility of studying the behaviour of MAP estimators in certain situations. In particular, we are able to prove a weak form of consistency of the MAP estimator in the infinitely informative data limit.

\subsection{Non-gaussian prior information and the need for MAP estimators}

A major challenge in Bayesian statistics is the extraction of information from the posterior distribution. 
In Gaussian-conjugate settings, such as linear inverse problems, 
 on the one hand there are explicit formulae for the posterior mean which can be used as an estimator of the unknown, and on the other hand one can (in principle) sample directly from a discretized version of the  posterior, \cite{AM84, LPS89, ALS13, PSZ13, PPRS12}.
Nevertheless, draws from Gaussian priors  do not vary sufficiently sharply on their domain to have a blocky structure, neither are they sparse.

For this purpose, the so-called TV prior has been used widely in applied literature, e.g. in medical imaging \cite{Setal03a, Setal03b}.
Drawing intuition from the classical regularization literature, the TV prior has a formal density of the form \[\pi(u)\propto \exp(-\alpha\norm{u}_{BV}),\] where $\alpha>0$ is a (hyper)parameter. Here, the norm of bounded variation, $\norm{u}_{BV}$, can be formally thought of as the $L^1$-norm of the derivative of $u$. However, the numerical implementation of Bayesian inversion with a TV prior miss-behaves as the discretization level of the unknown increases, and in particular the TV prior is not discretization invariant, \cite{LS04, LSS09, KLNS12}. For example, depending on the choice of the parameter $\alpha$ as a function of the discretization level, the posterior mean either diverges or converges to an estimate corresponding to a Brownian bridge prior.
For alternative types of approaches to edge-preserving non-parametric Bayesian inversion, see \cite{dunlop2015map, HL11, Helin09}.

We consider the family of Besov-priors which were proposed and shown to be discretization invariant in \cite{LSS09}. A well-posedness theory of the posterior was developed in \cite{Dashti_besov}. These priors are defined by a wavelet expansion with random coefficients, motivated by a formal density of the form 
\begin{equation}\label{eq:prd}\pi(u)\propto \exp(-\norm{u}_{B_{p}^s}^p),
\end{equation}
where $\norm{\cdot}_{B_{p}^s}=\norm{\cdot}_{B_{pp}^s}$ is the Besov space norm with regularity parameter $s$ and integrability parameters $p$. For $p=2$ the Besov space $B_{2}^s$ corresponds to the Sobolev space of functions with $s$ square-integrable derivatives, $H^s$, and the corresponding family of priors are Gaussian with Sobolev-type smoothness parametrized by $s$. We are especially interested in the case $p=1$, which is highly interesting for edge-preserving and sparsity-promoting Bayesian inversion, \cite{KLNS12}. For the case $s=1$ this is due to the close resemblance between the way $\norm{u}_{B_{1}^1}$ and $\norm{u}_{BV}$ work, both dealing with the $L^1$-norm of a (generalized) derivative of $u$, see \cite[Section 2]{KLNS12}. 
We define rigorously the class of Besov priors for $p=1$ and $s\in\R$, called $B^s_1$-Besov priors, in section \ref{sec:besov-def} below. 



To probe the posterior in the non-conjugate context of linear or nonlinear inverse problems with Besov priors, 
one typically resorts to Markov chain Monte Carlo (MCMC) methods. Unfortunately, in practice standard MCMC algorithms become prohibitively expensive for large scale inverse problems, consider e.g. photo-acoustic tomography \cite{pulkkinen16,tick16}. This is due to the heavy computational effort required for solving the forward problem which is needed for computing the acceptance probability at each step of the chain.

In such situations maximum a posteriori estimates are computationally attractive as they only require solving a single optimization problem. Furthermore, it was shown in \cite{KLNS12} that for the $B^1_1$-Besov prior defined via wavelet bases, at finite discretization levels the resulting MAP estimators are sparse and preserve the locations of the edges in the unknown. 
This remains true as discretization is refined, while
our work validates that, unlike in the TV prior case, the discretized MAP estimators converge to the MAP estimators of the limiting infinite dimensional posterior distribution.



%
%
%
%

\subsection{Review of the main results}\label{ssec:mapint}
Our results draw inspiration from previous papers by the authors \cite{Dashti13, HB15} (see also \cite{dunlop2015map}), where concepts that we will call  \emph{strong} and  \emph{weak MAP estimates} were coined.  
We quote both definitions in section \ref{sec:mapdefn} below.

As mentioned earlier, we can formally think of a MAP estimator as a mode of the posterior. In finite dimensional contexts, especially when working with continuous probability distributions, the definition of a mode is straightforward as a maximizer of the probability density function.  If the prior has probability density function of the form  \[\pi(u)\propto\exp(-W(u)),\]
for a suitable positive function $W:X\to [0,\infty)$, the density of the posterior is \[\dpost(u)\propto \exp(-I(u;y)),\]
where  \begin{equation}\label{eq:tik}I(u;y)=\Phi(u;y)+W(u).\end{equation} In this case, a MAP estimator can be interpreted as a classical estimator of the unknown, arising from Tikhonov regularization with penalty term given by the negative log-density of the prior, \cite{EHN96}. 

In the infinite dimensional setting things are less straightforward due to the lack of a uniform reference measure. An intuitive approach to define a mode in a function space $X$ is as follows:
compute the measure of balls with any center $u\in X$ and a fixed radius $\epsilon>0$ 
and proceed by letting $\epsilon$ tend to zero. A mode $\hat u$ is a center point  maximizing these small ball probabilities asymptotically (as $\epsilon$ decreases) in a specific sense.
What distinguishes a strong mode from a weak mode is exactly how `asymptotic maximality' is perceived: 
\begin{enumerate}
\item[a)] in the strong mode case, we look for the maximum probability among all centres in $X$;
\item[b)] in the weak mode case, we look for a centre of a ball with the property that it has maximum probability among all shifts of the ball by elements of a dense subspace $E\subset X$. 
\end{enumerate} 
Natural choices of $E$ turn out to be spaces of zero probability, thus giving one interpretation for the term `weak'. Note that weak and strong modes coincide for $E=X$.
Each of the two notions of mode gives rise to a notion of MAP estimator termed strong and weak MAP estimators, respectively.

Let $B_\epsilon(z)\subset X$ denote the open ball of radius $\epsilon$, centered at $z\in X$. If we can find a functional $I$ defined on an appropriate dense subspace $F\subset X$, such that 
\begin{equation}
\label{e:gOM}\lim_{\epsilon\to0}\frac{\mu(B_\epsilon(z_2))}{\mu(B_\epsilon(z_1))}=\exp(I(z_1)-I(z_2)), 
\end{equation} 
when $\mu=\post$, then for any fixed $z_1\in F$, a $z_2\in F$ maximising the limit or equivalently minimising $I$ is a potential MAP estimator.  
For weak MAP estimators, if the space $E\subset F$ over which we shift the ball centres, can be chosen sufficiently regular so that 
\begin{equation}\label{e:gwlim}
\lim_{\epsilon\to 0}\frac{\mu(B_\epsilon(u-h))}{\mu(B_\epsilon(u))}\quad\mbox{is continuous in $u\in X$ for any fixed $h\in E$},
\end{equation}
again for $\mu=\post$, then it is straightforward to establish the equivalence of weak MAP estimators and the minimisers of $I$. In the case of strong MAP estimators one needs to work considerably more, the difficulty stemming from the `smallness' of the subspace $F\subset X$ with respect to the prior hence also the posterior.

If a functional $I$ satisfying (\ref{e:gOM}) exists, it is called the (generalised) Onsager-Machlup  functional, \cite{DB78, IW14, Dashti13, HB15}. 
As for the limit in (\ref{e:gwlim}), it is identified with $R^\mu_h$, the Radon-Nikodym derivative of the shift of  $\mu$ by $h\in E$ with respect to $\mu$ itself, provided this derivative has a continuous representative, see \cite[lemma 2]{HB15} (quoted in lemma \ref{lem:cont_repres} below).
In finite dimensions, both the existence of the Onsager-Machlup functional and the statement in (\ref{e:gwlim}) follow from the Lebesgue differentiation theorem under very mild conditions on $\Phi$ and the density of the prior.
In the Banach-space setting, establishing (\ref{e:gOM}) and (\ref{e:gwlim}) for the posterior boils down to showing similar results for the prior given that the posterior is obtained by a suitably regular transformation of the prior.

Strong MAP estimators were defined and studied in \cite{Dashti13}, in the context of nonlinear Bayesian inverse problems with Gaussian priors. In this case, an expression for the limit in \eqref{e:gOM} was readily available from the Gaussian literature for $F$ being the Cameron--Martin space of $\mu_0$. The identification of minimisers of the resulting Onsager--Machlup functional with strong MAP estimators, however, required a considerable amount of work, in particular many new estimates involving small ball probabilities under the prior. 

Weak MAP estimators were defined and studied in \cite{HB15}, in the context of linear Bayesian inverse problems with a general class of priors. The authors used the tools from the differentiation and quasi-invariance theory of measures, developed by  Fomin and Skorohod \cite{Boga10}, to connect the zero points of $\beta_h^\mu$, the logarithmic derivative of a measure $\mu$ in the direction $h$,  to the minimisers of the Onsager-Machlup  functional. An essential assumption that makes this possible is the continuity of $\beta_h^\mu$ over $X$ for sufficiently regular $h$. The authors considered as examples Besov priors with integrability parameter $p>1$ and a conditionally Gaussian hierarchical prior with a hyper-prior on the mean.

In both \cite{Dashti13} and \cite{HB15}, the Onsager-Machlup functional was shown to be a Tikhonov-type functional as in \eqref{eq:tik}, where $W$ is the formal negative log-density of the prior, hence MAP estimators are identified with the corresponding Tikhonov approximations in the studied contexts. 


In this work, we are interested in defining and studying both the strong and the weak MAP estimates for generally nonlinear inverse problems with $B^s_1$-Besov priors, that is for Besov priors with integrability parameter $p=1$. Since the prior has formal density as in \eqref{eq:prd}, we expect and indeed show that MAP estimates are identified with minimizers of the Tikhonov-type functional 
\begin{equation}
\label{eq:besovOM}I(u;y)=\Phi(u;y)+\norm{u}_{B_1^s}.
\end{equation} 
In particular, the weak and strong MAP estimates coincide. 
For the considered prior,
the general theory developed in \cite{HB15} for weak MAP estimators does not apply, due to the fact that the logarithmic derivative of the $B^s_1$-Besov priors is inherently discontinuous. 
We will show that the continuity of the Radon-Nikodym derivative $R^\mu_h$, for $h$ in a suitable subspace $E\subset X$, is sufficient to get the result for weak MAP estimators. 
For the $B^s_1$-Besov priors, we prove the continuity of $R^\mu_h$ for shifts $h\in E\subseteq B^r_1(\T^d)$ for any $r>s$, and establish the validity of (\ref{e:gOM}) over $F=B^s_1(\T^d)$.
 It is then straightforward to prove our first main result which is the identification of weak MAP estimators by the minimisers of the functional $I$ in \eqref{eq:besovOM}.
 Furthermore, we generalise the program of \cite{Dashti13} to the assumed non-Gaussian case, in order to prove our second main result which is the identification of strong MAP estimators to minimisers of $I$ in \eqref{eq:besovOM}.
 We also consider the theory of local weak modes separately, relying on the Fomin literature.


\subsection{Consistency of MAP estimators}
Under the frequentist assumption of data generated from a fixed underlying true $u^\dagger$, it is desirable to verify that in the infinitely informative data limit, the Bayesian posterior distribution $\post$ contracts optimally to a Dirac distribution centered on $u^\dagger$. In recent years, there have been many studies on the rates of posterior contraction in the context of Bayesian inverse problems. The case of linear inverse problems with Gaussian and conditionally Gaussian priors is now well understood 
\cite{KVZ12, ALS13, KVZ13, ASZ14, KLS15,SVZ13,knapik2016}, and a theory for linear inverse problems with non-Gaussian priors is also being developed \cite{KR13, KS14}. For nonlinear inverse problems the asymptotic performance of the posterior is not yet fully understood, with some partial contributions being \cite{PSZ13, WZ16,vollmer2013}. 

Note that for general nonlinear problems, especially with finite dimensional data as in the present paper, one cannot expect to recover the underlying truth $u^\dagger$ in the infinitely informative data limit. Instead the aim is to recover a $u^\ast\in X$ such that $\G(u^\ast)=\G(u^\dagger)$. {Posterior contraction rates for Bayesian inverse problems with Besov priors are studied in ongoing work of a subset of the authors.} A form of weak consistency of the strong MAP estimator in the presence of repeated independent observations, for general nonlinear inverse problems with Gaussian priors, was shown in \cite{Dashti13}. 
In the present paper, we prove a similar result for the strong (hence the weak) MAP estimator obtained using the Besov prior for $p=1$. 

\subsection{Notation}

Throughout the paper we assume that $X$ is a separable Banach space equipped with the Borel $\sigma$-algebra. All probability measures are assumed to be Borel measures. The Euclidean norm in $\R^J$ is denoted by $|\cdot|$ to distuinguish it from the norm of any general $X$ denoted by 
$\norm{\cdot}_X$. We write $f\propto g$ for two functions $f,g:X\to \R$ if there exists a universal constant $c\in\R$ such that $f = cg$ as functions.

\begin{definition}\label{d:quasiinvariant} A measure $\mu$ is called \emph{quasi-invariant along $h$}, if the translated measure
$\mu_h(\cdot) := \mu(\cdot-h)$
is absolutely continuous with respect to $\mu$. We define \[Q(\mu)=\{h\, |\, \mu \;\text{is quasi-invariant along}\;h\},\]
which is readily verified to be a linear subspace.
\end{definition}

\begin{notation}
Let $h\in Q(\mu)$. We denote the Radon--Nikodym derivative of $\mu_h$ with respect to $\mu$ by $R_h^\mu \in L^1(\mu)$. 
\end{notation}


\subsection{Organization of the paper}
This paper is organized as follows: in section \ref{sec:mapdefn} we discuss the definition of modes for probability measures on separable Banach spaces. We introduce novel localized versions of the modes studied in previous work and discuss briefly how different modes can be characterized for log- or quasi-concave measures. The Besov priors are introduced and discussed in section \ref{sec:besov-def} including the key results relating to the Radon--Nikodym derivative of the Besov prior in section \ref{ssec:rnsb}. Section \ref{sec:bip} covers the Bayesian inverse problem setup and our main results related to identification of weak and strong MAP estimates as the minimizers of certain variational problem. Moreover, the weak consistency result is given in section \ref{sec:consistency}.
In section \ref{sec:logder} we discuss the logarithmic derivative of the posterior and its use in characterizing the MAP estimates. Finally, all proofs are postponed to section \ref{sec:proofs}.

\section{Modes of measures on Banach spaces}\label{sec:background}\label{sec:mapdefn}
In subsection \ref{ssec:wsgl} we introduce the two existing notions of {maximum a posteriori estimator} (modes of the posterior measure) proposed in
\cite{Dashti13, HB15} in the context of measures on infinite-dimensional spaces. We also define two new notions of local modes and hence local MAP estimates. In subsection \ref{ssec:logconc}, we focus on log-concave measures, study the structure of the set of modes, and give conditions for local modes to be global.

\subsection{Weak and strong, global and local}\label{ssec:wsgl}

The following definition of a mode, introduced in \cite{Dashti13}, grows out of the idea that highest small ball probabilities are obtained asymptotically at the mode.
\begin{definition}
\label{def:map}
Let
$M^\epsilon = \sup_{u\in X} \mu(B_\epsilon(u))$.
We call a point $\hat u\in X$ a \emph{mode of the measure $\mu$}, if it satisfies
$$\lim_{\epsilon\to 0}\frac{\mu(B_\epsilon(\hat u))}{M^\epsilon} = 1.$$ A mode of the posterior measure $\post$ in \eqref{eq:intpost}, is called a \emph{maximum a posteriori} (MAP) estimate. 
\end{definition}
Below we occasionally use the terms \emph{strong mode} and \emph{strong MAP} estimator for the concepts introduced in definition \ref{def:map} 
in order to distinguish them from the following weaker notion of a mode (similarly, the \emph{weak mode} or \emph{weak MAP}) introduced in \cite{HB15}.  

\begin{definition}
\label{def:wmap}
Let $E$ be a dense subspace of $X$. We call a point $\hat u \in X$, $\hat u \in {\rm supp} (\mu)$, a \emph{weak mode} of $\mu$ if 
\begin{equation}
	\label{eq:wmap}
	\lim_{\epsilon \to 0} \frac{\mu(B_\epsilon(\hat u - h))}{\mu(B_\epsilon(\hat u))} \leq 1,
\end{equation}
for all $h\in E$. A weak mode of the posterior measure $\post$ in \eqref{eq:intpost}, is called a \emph{weak maximum a posteriori} (wMAP) estimate.
\end{definition}

Notice that the definition of a weak mode is dependent on the choice of the subspace $E$. Therefore, in some contexts it may be more appropriate to discuss $E$-weak modes. In the following, however, we will suppress this dependence since the key question to our study is whether such a space exists.

The notions of weak and strong mode are related as follows: any strong mode is a weak mode for the choice $E=X$ \cite[lemma 3]{HB15}, which is straightforward to see by simply estimating $\mu(B_\epsilon(\hat u - h)) \leq M^\epsilon$.
The key motivation to study the weak definition, is the case when small ball asymptotics are not available explicitly or only available in some subspace of translations $h$. 
It is then of interest to choose $E$ so that an expression for the limit on the left hand side of \eqref{eq:wmap} exists pointwise.
This typically leads to choices of $E$ which have zero probability with respect to $\mu$.
The following lemma, which is an immediate generalization of \cite[lemma 2]{HB15}, provides further guidance for this choice.

\begin{lemma}
\label{lem:cont_repres}
Assume that $\mu$ is quasi-invariant along the vector $h$. Let $A\in\mathcal{B}(X)$ be convex, bounded and symmetric and define $A^\diam:=\diam A$. Suppose $R_h^\mu$ has a continuous representative $\tilde R_h^\mu \in C(X)$, i.e., $R_h^\mu - \tilde R_h^\mu = 0$ in $L^1(\mu)$.  Then it holds that
\begin{equation*}
	\lim_{\diam \to 0} \frac{\mu_h(A^\diam+u)}{\mu(A^\diam+u)} = \tilde R_h^\mu(u)
\end{equation*}
for any $u \in X$.
\end{lemma}

\begin{remark}\label{rem:cont_repres} According to lemma \ref{lem:cont_repres}, it is desirable to consider a subspace $E$, such that $R_h^\mu$ is continuous for $h\in E$. A sufficient condition for the continuity of $R_h^\mu$ was given in \cite{HB15}, namely the so-called logarithmic derivative of $\mu$ along $h$ needed to be continuous and exponentially integrable with respect to $\mu$.
For the sparsity promoting measure that we will consider in this paper, the logarithmic derivative is inherently discontinuous (section \ref{sec:logder}). We are however able to show the continuity of $R_h^\mu$ over an appropriate subspace $E$ in subsection \ref{ssec:rnsb} by using its explicit expression.
\end{remark}


Both strong and weak mode can also be associated with a natural localization described by the following definitions:

\begin{definition}[Local modes]
Let $\hat u \in X$ be such that $\hat u \in {\rm supp} (\mu)$.
\begin{enumerate} 
\item[(1)] We call $\hat u$ a {\em local mode of the measure} $\mu$, if there exists a $\delta>0$ such that the quantity $M^\epsilon_\delta = \sup_{u\in B_\delta(\hat u)} \mu(B_\epsilon(u))$ satisfies
\begin{equation*}
	\lim_{\epsilon\to 0} \frac{\mu(B_\epsilon(\hat u))}{M^\epsilon_\delta} = 1.
\end{equation*}
\item[(2)]We call $\hat u$ a \emph{local weak mode} of $\mu$ if 
there exists $\delta>0$ such that
\begin{equation}
	\label{eq:lwmap}
	\lim_{\epsilon \to 0} \frac{\mu(B_\epsilon(\hat u - h))}{\mu(B_\epsilon(\hat u))} \leq 1
\end{equation}
for all $h\in B^X_\delta(0) \cap E$. 
\end{enumerate}
\end{definition}

Local modes represent an analogue to local maxima of the probability density function in finite dimensions.
In the setting of \cite{HB15} the local wMAP coincides with the zero points of the logarithmic derivative of the posterior.
Especially, regularization techniques in non-linear inverse problems are often known to give birth to local maxima
and, therefore, the local wMAP can give statistical interpretation for these points.

\subsection{Modes for log-concave measures}\label{ssec:logconc}

This work studies the Besov prior which is a prototypical example of a log-concave measure. Below we show some general properties regarding the modes of such class of measures. Most of these ideas naturally extend to a larger class called the quasi-concave measures.

\begin{definition}\label{d:kconc}
	A probability measure $\mu$ in $(X, \mathcal{B}(X))$ is called \emph{logarithmically-concave} or \emph{log-concave}, if 
	\[\mu(\lambda A+(1-\lambda)B)\ge \mu(A)^\lambda \mu(B)^{1-\lambda},\] 
	for all $A,B\in \mathcal{B}(X).$ Moreover, $\mu$ is called \emph{quasi-concave} if
	\[\mu(\lambda A+(1-\lambda)B)\ge \min\{\mu(A), \mu(B)\},\] 
	for all $A,B\in \mathcal{B}(X).$
\end{definition}
It is straightforward to verify that any log-concave measure is also quasi-concave.
An immediate result of quasi-concavity is the well-known Anderson inequality \cite{TA55, CB74}.

\begin{proposition}\label{c:B-anderson}
Let $\mu$ be a symmetric quasi-concave measure on $X$. For any symmetric and convex set $A\subset X$ we have 
$$
\mu(A+x)\le \mu(A),\quad \mbox{for any } x\in X.
$$
\end{proposition} 

The next result follows from the Anderson inequality.

\begin{proposition}\label{prop:symqcex}
Suppose that the measure $\mu$ on $X$ is symmetric around $u$ and quasi-concave. Then $u$ is a strong mode of $\mu$.%
\end{proposition}

Let us next consider briefly the structure of the set of strong modes of a quasi-concave measure $\mu$. 
When working in finite dimensions, $X=\R^d$, a probability density function $f$ with respect to the Lebesgue measure
is called quasi-concave if for all $x,y\in\R^d$ and all $\lambda\in[0,1]$ we have \[f((1-\lambda)x+\lambda y)\geq \min\{f(x),f(y)\}.\]
Clearly a quasi-concave probability density function has a convex set of global modes. For a reference on convexity and unimodality in finite dimensions see \cite{DJ88}. We show that a similar result holds in infinite dimensions for our definition of strong mode.

\begin{proposition}\label{prop:qcunim}
Suppose that the measure $\mu$ in $X$ is quasi-concave (but not necessarily symmetric). Then the set of strong modes is convex.
\end{proposition}

It turns out log-concavity is a sufficient condition for the global and local modes to coincide. 

\begin{theorem}\label{thm:wmode}
Suppose $\mu$ is log-concave and $\hat u$ is a local mode. Then $\hat u$ is also a global mode.
Similarly, if $\hat u$ is a local weak mode then it is also a global weak mode.
\end{theorem}

\section{Besov priors with $p=1$}
\label{sec:besov-def}
The family of Besov priors has been introduced in \cite{LSS09} and studied in \cite{Dashti_besov, HB15}. In subsection \ref{ssec:defn} we recall the definition and some useful properties of Besov priors with integrability parameter $p=1$ and regularity parameter $s>0$, termed $B^s_1$-Besov priors, on which we focus in this work. We also present some straightforward convexity properties of $B^s_1$-Besov priors. 
The main results of this section are listed in subsection \ref{ssec:rnsb}, where we compute the Radon-Nikodym derivative $R_h^\mu$ for $B^s_1$-Besov priors, determine the space $E$ in which $h$ needs to live in order for the corresponding Radon-Nikodym derivative $R_h^\mu$ to be continuous, and finally show that $I_0(u)=\|u\|_{B^s_1}$ is the Onsager-Machlup functional for $B^s_1$-Besov measures.

\subsection{Definition and basic properties}\label{ssec:defn}

We work with periodic functions on a $d$-dimensional torus, $\T^d$. We first define the periodic Besov spaces $B^s_{pq}(\T^d)$, where $s\in\R$ parametrises smoothness and $p,q\geq 1$ are integrability parameters. We concentrate on the case $p=q$ and write $B^s_p=B^s_{pp}$. To define the Besov spaces, we let $\{\psi_\ell\}_{\ell=1}^\infty$ be
an orthonormal wavelet basis 
for $L^2(\T^d)$, where we have utilized a global indexing.
We can then characterise $B^s_{p}(\T^d)$  using the given basis
in the following way: the function $f:\T^d\to\R$ defined by the series expansion
\begin{equation}
	f(x) = \sum_{\ell=1}^\infty c_\ell \psi_\ell(x)
\end{equation}
belongs to $B^s_{p}(\T^d)$, if and only if the norm
\begin{equation}
	\norm{f}_{B^s_{p}(\T^d)}
	= \left(\sum_{\ell=1}^\infty \ell^{p(\frac {s}d+\frac 12)-1} |c_\ell|^p\right)^\frac1p,
\end{equation}is finite.
Throughout, we assume that the basis is $r$-regular for $r$ large enough in order to consist a basis for a Besov space with smoothness $s$, \cite{Daubechies}. 

We now follow the construction in \cite{LSS09} to define periodic Besov priors corresponding to $p=1$,
using series expansions in the above wavelet basis with random coefficients. Notice also the work \cite{Ramlau} on defining Besov priors for functions on the full space $\R^d$.
\begin{definition}
\label{def:besov_prior}
Let $(X_\ell)_{\ell=1}^\infty$ be independent identically
distributed real-valued random variables with the probability density function
\begin{equation}
	\pi_X(x) = \frac 12 \exp(-|x|).
\end{equation}
Let $U$ be the random function
\begin{equation*}
	U(x) = \sum_{\ell=1}^\infty \ell^{-{(\frac sd-\frac 12)}} X_\ell \psi_\ell(x),
	\quad x \in \T^d.
\end{equation*}
Then we say that $U$ is distributed according to a $B^s_1$-Besov prior.
\end{definition}
The next lemma determines the smoothness of functions drawn from the $B^s_1$-Besov prior and shows the existence of certain exponential moments.
\begin{lemma}\cite[lemma 2]{LSS09}
\label{lem:besov_lss}
Let $U$ be as in Definition \ref{def:besov_prior} and let $t<s-d$. Then it holds that
\begin{itemize}
	\item[(i)] $\norm{U}_{B^t_{1}} < \infty$, \quad almost surely, and
	\item[(ii)] $\expec \exp(\frac 12 \norm{U}_{B^t_{1}} ) < \infty$.
\end{itemize}
\end{lemma}

\begin{notation}\label{not:densities}We denote by $\rho_\ell$ the probability measure of the random variable $\ell^{-s/d-1/2}X_\ell$ on $\R$.
We identify the random function $U$ in definition \ref{def:besov_prior} with the product measure of coefficients $(\ell^{-s/d-1/2}X_\ell)_\ell$ in $(\R^\infty,{\mathcal B}(\R^\infty))$, which we denote by $\lambda=\bigotimes_{\ell=1}^\infty \rho_\ell$.\end{notation}

We next consider the convexity of the $B^s_1$-Besov prior.
\begin{lemma}\label{lem:logconc}
For any $s>0$, the $B^s_1$-Besov measure $\lambda$ is logarithmically concave.
\end{lemma}

An immediate consequence of the last lemma, is that by proposition \ref{c:B-anderson}, the $B^s_1$-Besov prior satisfies Anderson's inequality. 
Notice that lemma \ref{lem:logconc} (and hence proposition \ref{c:B-anderson}) also holds for $B^s_p$-Besov measures with $p>1$, as defined in \cite{LSS09}. 


\subsection{Radon--Nikodym derivative $R^\mu_h$ and small ball probabilities}\label{ssec:rnsb}

Recall the definitions of quasi-invariance for a measure $\mu$ and of the subspace of directions in which $\mu$ is quasi-invariant, $Q(\mu)$.  In general the structure of $Q$ is not known and there are even examples of measures for which $Q$ fails to be locally convex \cite[exercise 5.5.2]{Boga10}. The space $Q$ is known to be a Hilbert space for certain families of measures, for example for $\alpha$-stable measures with $\alpha\geq 1$ and for countable products of a single distribution with finite Fisher information, see \cite[theorem 5.2.1]{Boga10}  and \cite{LS65} (note that the $B^s_1$-Besov prior, $\gl$, is not an $\alpha$-stable measure). Using similar techniques to \cite{LS65}, namely the Kakutani--Hellinger theory, we now show that $Q$ is a Hilbert space also for the $B^s_1$-Besov measure $\lambda$ and calculate the Radon-Nikodym derivative $R_h^\gl$ between $\gl$ and the shifted measure $\gl_h$ for $h\in Q(\gl)$.
\begin{lemma}\label{l:BesovRN}
For the $B^s_1$-Besov measure we have $Q(\gl)=B_2^{s-\frac{d}2}(\T^d)$. For $h\in Q(\gl)$ we have
$$
\frac{d\gl_h}{d\gl}(u)=\lim_{N\to\infty}\exp\sum_{\ell=1}^N\big( -\ga_\ell|h_\ell-u_\ell|+\ga_\ell|u_\ell| \big)
$$
in $L^1(\R^\infty,\gl)$, where $\ga_\ell=\ell^{s/d-1/2}$.
\end{lemma}

We next provide a more detailed view of spaces of shifts $h$ for which $R_h^\gl$ has a continuous representative,
which we denote by $\tr_h^\gl(u)$. This is a crucial result for our study of weak MAP estimators, see lemma \ref{lem:cont_repres}, Remark \ref{rem:cont_repres} and the discussion in subsection \ref{ssec:mapint}.

\begin{lemma}\label{l:crh}
Let $h=\sum_{\ell\in\N}h_\ell\psi_\ell\in B^{\gs}_1(\T^d)$ with $\gs>s$. Then 
$\tr_h^\gl(u)=\exp\sum_{\ell=1}^\infty\big( -\ga_\ell|h_\ell-u_\ell|+\ga_\ell|u_\ell| \big)$ is continuous with respect to $u=\sum_{\ell\in\N}u_\ell\psi_\ell\in B^t_1(\T^d)$ for any $t<s-d$.
\end{lemma}
Note that in the expression for the Radon-Nikodym derivative $\tr_h^\gl(u)$, $h\in E$ and the less regular $u\in X$ are coupled component-wise (in the wavelet basis defining the Besov measure) and hence establishing the continuity of $\tr_h^\gl(u)$ with respect to $u$ is not straightforward. See section \ref{sec:proofs} for the proof of above lemma. 

We record the following immediate corollary of the last lemma and lemma \ref{lem:cont_repres}.
\begin{corollary}\label{cor:BesovSubopt}
Let $h\in B^\gs_1(\T^d)$, $\gs>s$. Then it holds that 
\begin{align}\label{e:c1}
\lim_{\epsilon\to0}\frac{\lambda_h(B_\epsilon(u))}{\lambda(B_\epsilon(u))}=\exp\sum_{\ell=1}^\infty(-\alpha_\ell|h_\ell-u_\ell|+\alpha_\ell|u_\ell|), 
\end{align}
for any $u\in B_1^t$.
\end{corollary}
\begin{remark}
Let $X=B^t_1(\T^d)$ for $t<s-d$. By remark \ref{rem:cont_repres} and noting that the Besov space $B^\gs_1(\T^d)$ is dense in $B^t_1(\T^d)$ for any $\gs>t$, the last corollary shows that it is natural to choose $E=B^\gs_1(\T^d)$ for any $\gs>s$ in the definition of wMAP estimate.
It also shows that the origin is the unique weak mode and, therefore, by proposition \ref{prop:symqcex} also the unique strong mode.
\end{remark}

We also remark that the series above consists of the difference of two terms, whose respective series are not convergent in general for $h\in B^\gs_1(\T^d)$. However, they coincide with the difference of the norms in $h\in B^\gs_1(\T^d)$ if $h$ and $u$ are elements of this smaller space.
Based on this view we close this section with an important building block for the study of the MAP estimate. It extends the last corollary to $h\in B^s_1(\T^d)$, for balls centered at the origin. 
\begin{theorem}\label{l:B1CM}
Suppose that $h\in B^s_1(\T^d)$ and $t<s-d$. Let $A\in\mathcal{B}(B^t_1(\T^d))$ denote a convex and zero-centered symmetric and bounded set. Then
$$
\lim_{\gep\to 0}\frac{\gl_h(\gep A)}{\gl(\gep A)}=\exp(-\|h\|_{B^s_1}).
$$
\end{theorem}

It follows immediately from the above lemma that for $z_1,z_2\in B^s_1(\T^d)$,
\begin{align}\label{e:c2}
\lim_{\epsilon\to0}\frac{\lambda(B_\epsilon(z_1))}{\lambda(B_\epsilon(z_2))}=\exp(-\|z_1\|_{B^s_1}+\|z_2\|_{B^s_1}),
\end{align}
giving the Onsager--Machlup functional of $\gl$. The space $B^s_1(\T^d)$ here, is the largest space on which the Onsager--Machlup functional is defined. This is the space $F$ of the discussion in subsection \ref{ssec:mapint} in the case of Besov priors.

It is also worth noting the difference between (\ref{e:c1}) and (\ref{e:c2}). The centres of the balls in (\ref{e:c1}) are in $X=B^t_1(\T^d)$, the continuity of the right-hand side in $u$ is due to sufficient regularity of the shift $h$. 
In (\ref{e:c2}), the centres of the balls are more regular than (\ref{e:c1}), but the shift is less regular in general, i.e.\ only in $B^s_1(\T^d)$.

\section{Main results}\label{sec:bip}

We are now ready to present our main results. We first show the existence of weak and strong MAP estimates in Bayesian inverse problems with $B^s_1$-Besov priors, and identify both with the minimisers of the Onsager--Machlup functional. Using this characterization, we then prove a weak consistency result for MAP estimators.

\subsection{Identification of MAP estimates for Bayesian inverse problems}

We consider the inverse problem of estimating a function $u \in X=B^t_1(\T^d)$ from a noisy and indirect observation $y\in \R^J$, modelled as 
\begin{equation}
	\label{eq:inverse_problem}
	y = \G(u) + \xi.
\end{equation}
Here $\G : B^t_1(\T^d) \to \R^J$ is a locally Lipschitz continuous, possibly non-linear operator and $\xi$ is Gaussian observational noise in $\R^J$, $\xi\sim N(0,\Sigma)$ for a positive definite covariance matrix $\Sigma\in \R^{J\times J}$.

 We assume that $u$ is distributed according to the $B^s_1$-Besov measure $\lambda$ defined in Section \ref{sec:besov-def}, so that $\lambda(B^t_1(\T^d))=1$ for $t<s-d$. Under the assumption of local Lipschitz continuity of $\G$, it follows that almost surely with respect to $y$ the posterior distribution $\mu^y$ on $u|y$, has the following Radon--Nikodym derivative with respect to the prior $\lambda$: 
\begin{equation}\label{e:muy}
	\frac{d\mu^y}{d\lambda}(u)=\frac{1}{Z(y)}\exp(-\Phi(u; y)),	
\end{equation}
where
\begin{equation}\label{e:Phi}
	\Phi(\cdot ; y) = \frac 12 \left|\Sigma^{-\frac 12} (y-\G(\cdot))\right|^{2} 
\end{equation} 
and $Z(y)$ is the normalization constant. Indeed, local Lipschitz continuity of $\G$ implies measurability of $\Phi(\cdot, y)$ with respect to $\lambda$; it also, together with non-negativity of $\Phi$, gives the finiteness and non-singularity of $Z$. For details see \cite{DS15}. 

Following the intuition described in subsection \ref{ssec:mapint}, we define the Tikhonov-type functional
\begin{equation}
	\label{eq:I}
	I(u; y) := \begin{cases}
	\Phi(u; y) + \norm{u}_{B^s_1}, & u \in B^s_1(\T^d), \\
	\infty & {\rm otherwise.}
	\end{cases}
\end{equation}
The existence of minimizers for $I$ is classical, however, we include the proof for completeness.
\begin{lemma}\label{l:OMmin}
The functional $I(\cdot; y)$ in equation \eqref{eq:I} has a minimizer $\hat u \in B^s_1(\T^d)$.
\end{lemma}

The underpinning of our main results in theorems \ref{t:OMwMAP} and \ref{c:sMAP} is that the posterior inherits the property of the prior given in lemma \ref{l:crh}, i.e. there exists a subspace, where the limit of the translated small ball probability ratios has a continuous representative.
Moreover, one can show how this limit is connected with functional $I$.
The next result follows directly by
combining lemmas \ref{lem:cont_repres}, \ref{l:crh} and \ref{l:B1CM}, and the local Lipschitz continuity of $\Phi$.

\begin{proposition}
\label{prop:limit_is_onsager}
Let $A\in\mathcal{B}(B^t_1(\T^d))$ be convex, bounded and symmetric and define $A^\diam:=\diam A$. 
\begin{itemize}
\item[i)] For any $h\in B^\gs_1(\T^d)$ with $\gs>s$ the mapping
$$
u\mapsto \lim_{\gep\to 0}\frac{\post_h(A^\diam+u)}{\post(A^\diam+u)}
$$
is a continuous function of $u\in B^t_1(\T^d)$.
\item[ii)]The Tikhonov-type functional defined in \eqref{eq:I} is the generalized Onsager--Machlup functional for the posterior $\post$ in \eqref{e:muy}. That is, for any $z_1, z_2\in B^s_1(\T^d)$, 
\[\lim_{\epsilon\to0}\frac{\post(A^\diam+z_2)}{\post(A^\diam+z_1)}=\exp(I(z_1;y)-I(z_2;y)).\]
\end{itemize}
\end{proposition}

The main theorems below show that the weak and strong MAP estimates of the posterior $\post$ in \eqref{e:muy}, are identified with minimisers of functional $I$. In particular, the weak and strong MAP estimates coincide for the inverse problems considered here. 
This is remarkable since it is not known in general under which conditions there exists weak MAP estimates that are not strong.

\begin{theorem}\label{t:OMwMAP}
An element $u \in B^s_1(\T^d)$ minimizes $I(\cdot; y)$ if and only if it is a weak MAP estimate for the posterior measure $\post$ in \eqref{e:muy}.
\end{theorem}

We note that the last result implies the existence of weak MAP estimates by lemma \ref{l:OMmin}.
We next show the existence of strong MAP estimators, and that any strong MAP estimate is a minimiser of $I$.

\begin{proposition}\label{t:sMAP}

Consider the measure $\mu^y$ given by (\ref{e:muy}) and (\ref{e:Phi}) with the $B^s_1$-Besov prior $\gl$ and $\G:B^t_1(\T^d)\to \R^J$ locally Lipschitz for $t<s-d$. 
\begin{itemize}
\item[i)] For any $\gd>0$ there exists $z^\gd\in B^t_1(\T^d)$ satisfying $z^\gd=\arg\max_{z\in X}\mu^y(A^\gd+z)$, where $A^\delta:=\delta A$ with $A$ a convex, symmetric and bounded set in $B^t_1(\T^d)$.
\item[ii)] There is a $\bar z\in B^s_1(\T^d)$ and a subsequence of $\{z^\gd\}_{\gd>0}$ which converges to $\bar{z}$ strongly in $B^t_1(\T^d)$.
\item[iii)] The limit $\bar z$ is a strong MAP estimator and a minimizer of $I(u;y)$ in equation \eqref{eq:I}.
\end{itemize}
\end{proposition}

In the following theorem we prove that any minimizer of $I$ is a strong MAP estimate and, hence, show the identification of strong MAP estimates and minimizers of $I$ using part (iii) of proposition \ref{t:sMAP}.

\begin{theorem}\label{c:sMAP}
Suppose that conditions of proposition \ref{t:sMAP} hold. Then the strong MAP estimators of $\mu^y$ are characterised by the minimizers of the Onsager--Machlup functional $I$ given in \eqref{eq:I}.
\end{theorem}

The proof of theorem \ref{t:OMwMAP} concerning weak MAP estimates is relatively straightforward and relies on lemma \ref{l:crh}, i.e., the ability to consider the subspace $B^\gs_1(\T^d)$, where the Radon--Nikodym derivative $R_h^{\mu^y}$ has a continuous representative.
The proof of proposition \ref{t:sMAP} related to strong MAP estimates is more involved and requires a series of technical lemmas (lemma \ref{l:bdrym} to \ref{l:limB-wc}) which are stated in section \ref{sec:proof_for_sec4}. Our approach is based on developing asymptotic estimates for the small ball probability ratios by projecting the related measures to finite dimensions, where explicit calculations can be carried out.

The difference in difficulty of the proofs related to the two notions of MAP estimates highlights the flexibility of weak MAP estimators. It seems that explicit calculations are typically required for the proof in the case of strong MAP estimates. For practical purposes, it is a highly interesting future task to develop general conditions under which the two MAP estimate concepts coincide.

\begin{remark}
Proposition 4.3.8 in \cite{Boga10} shows that if $\Phi$ is convex in $u$, then since by lemma \ref{lem:logconc} the $B^s_1$-Besov prior $\lambda$ is logarithmically-concave, the posterior $\mu^y$ is also logarithmically-concave and hence quasi-concave. In that case proposition \ref{prop:qcunim}  shows that the set of modes is convex. The convexity of $\Phi$ depends on the forward operator $\G$: if for example $\G$ is linear then $\Phi$ is convex, however in the general nonlinear case $\Phi$ may be non-convex.
\end{remark}
%
%

\subsection{Weak consistency of the strong MAP}\label{sec:consistency}

We consider the frequentist setup, in which \[y_j\stackrel{i.i.d.}{\sim}N(\G(u^\dagger),\Sigma), \;\;j=1,\dots,n,\] for a  fixed underlying value of the unknown functional parameter $\utr\in X$. As before, we assume  that $\G:X\to\R^J$  is locally Lipschitz and $\Sigma\in \R^{J\times J}$ is a positive definite matrix.

For this set of data and with a $B^s_1$-Besov prior, $\gl$, the posterior measure satisfies
\begin{align}\label{e:muyn}
	\frac{d\mu^{y_1,y_2,\dots,y_n}}{d\lambda}(u) \propto \exp\Big(-\frac{1}{2}\sum_{j=1}^n \big|\Sigma^{-\frac{1}{2}}(y_j-\G(u))\big|^2\Big).	
\end{align}
Here also the Lipschitz continuity of $\G$ implies the well-definedness $\mu^{y_1,y_2,\dots,y_n}$ \cite{DS15}. Proposition \ref{t:sMAP} then implies that the strong MAP estimator of the above  posterior measure is a minimizer of 
\begin{align}\label{e:In}
I_n(u):=\|u\|_{B^s_1}+\frac{1}{2}\sum_{j=1}^n \big|\Sigma^{-\frac{1}{2}}(y_j-\G(u))\big|^2.
\end{align}
\begin{theorem}\label{t:consistency}
Suppose that $\G:B^t_1(\T^d)\to\R^J$ is locally Lipschitz and $\utr\in B^s_1(\T^d)$. Denote the minimizers of $I_n$ given in (\ref{e:In}) by $u_n$ for each $n\in N$. Then there exists $u^*\in B^s_1(\T^d)$ and a subsequence of $\{u_n\}$ such that $u_n\to u^*$ in $B^\ts_1(\T^d)$ almost surely for any $\ts<s$.
For any such $u^*$ we have $\G(u^*)=\G(\utr)$. 
\end{theorem}

If $\utr$ lives only in $X=B^t_1(\T^d)$, and not necessarily in $B^s_1(\T^d)$, we can only get the convergence of $\{\G(u_n)\}$:

\begin{corollary}\label{c:consistency}
Let $\G$ and $u_n$, $n\in\N$, satisfy the assumptions of theorem \ref{t:consistency} and suppose that $\utr\in B^t_1(\T^d)$.
Then there exists a subsequence of $\{\G(u_n)\}_{n\in\N}$ converging to $\G(\utr)$ almost surely.
\end{corollary}

Theorem \ref{t:consistency} states that the true solution is identified in the range of $\G$, which is the natural objective also in regularization theory \cite{EHN96}. The full identification of $\utr$ is dependent on further properties of $\G$, e.g., injectivity of $\G$ would immediately yield $u^* = \utr$.

\section{Connections to logarithmic derivative}\label{sec:logder}

In this section we discuss the logarithmic derivative of the posterior measure. We mainly revisit known results (see e.g. \cite{Boga10}) and also derive the logarithmic derivative of the posterior $\mu^y$ given in \eqref{e:muy}. The intuition behind logarithmic derivative is that it roughly corresponds to the G\^{a}teaux derivative of the posterior potential $I$. 
If the logarithmic derivative is smooth, then its zero points can determine the weak MAP estimates as shown in \cite{HB15}. The Besov $B^s_1$-prior does not meet this criteria due to the discontinuity of its logarithmic derivative at the origin as we show in 
theorem \ref{t:Besov-logder}. The logarithmic derivative also determines the Radon--Nikodym derivative as recorded in proposition \ref{prop:onsager} below. Moreover, it can be used as a basis of Newton-type algorithms to estimate the weak MAP in case an explicit form of the potential $I$ is not easily accessible, see e.g. Cauchy priors in \cite{sullivan2016well}.

\begin{definition}
A measure $\mu$ on $X$ is called Fomin differentiable along the vector $h$ if, for every set ${\mathcal A}\in {\mathcal B}(X)$, there exists a finite limit
\begin{equation}
	\label{eq:dhmu_def}
	d_h\mu({\mathcal A}) = \lim_{t\to 0} \frac{\mu({\mathcal A}+th)-\mu({\mathcal A})}{t}
\end{equation}
\end{definition}
It is well-known that if $\mu$ is Fomin differentiable along $h$ then the limit $d_h\mu$ is a countably additive signed measure on ${\mathcal B}(X)$ and has bounded variation \cite{Boga10}. Moreover, $d_h \mu$ is absolutely continuous with respect to $\mu$.

We denote the domain of differentiability by
\begin{equation}
	D(\mu) = \{h \in X \; | \; \mu \textrm{ is Fomin differentiable along } h\}
\end{equation}

\begin{definition}
The Radon--Nikodym density of the measure $d_h\mu$ with respect to $\mu$ is denoted
by $\beta^\mu_h$ and is called the logarithmic derivative of $\mu$ along $h$.
\end{definition}

\begin{proposition}\cite[prop. 6.4.1]{Boga10}
\label{prop:onsager}
Suppose $\mu$ is a Radon measure on a locally convex space $X$ and is Fomin differentiable along a vector $h\in X$.
If it holds that $\exp(\epsilon |\beta^\mu_h(\cdot)|) \in L^1(\mu)$ for some $\epsilon>0$, then 
$\mu$ is quasi-invariant along $h$ and the Radon--Nikodym density $R_h^\mu$ of $\mu_h$ with respect to $\mu$
satisfies the equality
\begin{equation}
	\label{eq:general_onsager}
	R_h^\mu(u) = \exp\left(\int_0^1 \beta^\mu_h(u-sh)ds\right) \quad {\rm in }\; L^1(\mu).
\end{equation}
\end{proposition}

\begin{remark}
Recall that as discussed in Remark \ref{rem:cont_repres}, it is desirable to choose $E$ in the definition of wMAP to be a subspace $E\subset X$ such that $R_h^\mu$, $h\in E$, has a continuous representative $\tilde R_h^\mu$. Therefore, the integral $\int_0^1\beta^\mu_h(u-sh)ds$ in \eqref{eq:general_onsager} has a measurable representative, which is continuous outside the set $\{u \in X \; | \; \tilde R_h^\mu(u)=0\}$. Moreover, a weak mode $\hat u$ of $\mu$ can be equivalently defined 
by condition
\begin{equation*}
	\int_0^1\beta^\mu_h(u-sh)ds \leq 0
\end{equation*}
for all $h\in E$.
\end{remark}

The construction of the Besov prior in definition \ref{def:besov_prior} is a prototypical example of a product measure.
By setting $\lambda_\ell = \frac{1}{a_\ell} \pi_X\left(\frac x{a_\ell}\right) dx$ for $a_n = \ell^{-\left(\frac sd - \frac 12\right)}$ we
can define the probability law of the Besov prior on $(\R^\infty,{\mathcal B}(\R^\infty))$ by
$\lambda = \otimes_{\ell=1}^\infty \lambda_\ell$.
For product measures, the Fomin differentiability calculus reduces to finite dimensional projections in a straightforward manner.
\begin{lemma}\cite[prop. 3.4.1 (iii)]{Boga10}
\label{lem:onedim_diff}
Let $\mu$ be a probability measure on $(\R, {\mathcal B}(\R))$. Then $\mu$ is Fomin
differentiable along $h\neq 0$ if and only if has an absolutely continuous density $\pi_\mu$ whose derivative satisfies $\pi_\mu' \in L^1(\R)$. In this case, $d_1 \mu = \pi_\mu' dx$.
\end{lemma}

The Fomin differentiability of a product measure $\mu = \otimes_{n=1}^\infty \mu_n$ on the space
$X = \prod_{n=1}^\infty X_n$ assigned with the product topology is characterized by the following theorem:
\begin{proposition}\cite[prop. 4.1.1.]{Boga10}
\label{thm:product_measures_conv}
Suppose that $\beta^{\mu_n}_{h_n}$ is the logarithmic derivative of $\mu_n$ in the direction $h_n\in X_n$. The following claims are equivalent:
\begin{itemize}
	\item[(i)] $\mu$ is differentiable along $h = (h_j)_{j=1}^\infty \in X$,
	\item[(ii)] the series $\sum_{n=1}^\infty \beta_{h_n}^{\mu_n}$ converges in the norm of $L^1(\mu)$ and
	$$\beta^\mu_h=\lim_{N\to\infty}\sum_{n=1}^N \beta_{h_n}^{\mu_n}.$$
\end{itemize}
\end{proposition}

Let us introduce the following subspace
\begin{equation*}
	H(\mu) = \{h \in D(\mu) \; | \; \beta^\mu_h \in L^2(\mu)\} \subset D(\mu),
\end{equation*}
which has a natural Hilbert space structure \cite[section 5]{Boga10}.
Surprisingly, for a large class of product measures $H(\mu)$ coincides with $D(\mu)$.
The following proposition follows from \cite[cor. 2]{BS90} and \cite[ex. 5.2.3]{Boga10}.
\begin{proposition}
\label{cor:prod_measures}
Suppose $\mu = \pi(x) dx$ is a Borel probability measure on the real line such that 
\begin{equation*}
	\int_\R \frac{\pi'(t)^2}{\pi(t)} dt < \infty.
\end{equation*}
If we set
$\mu_n(A) = \mu(A/a_n)$, where $a_n>0$, and $\mu = \otimes_{n=1}^\infty \mu_n$, then it follows that
\begin{equation*}
	D(\mu) = H(\mu) = \left\{h \in \R^\infty \; \left| \; \sum_{n=1}^\infty a_n^{-2} h_n^2 < \infty \right\}.\right.
\end{equation*}
\end{proposition}

Let us record the following direct consequence of proposition \ref{prop:onsager}: if $t>0$ then
\begin{equation}
	\label{eq:onsager_reparam}
	R_{th}(u) = \exp\left(\int_0^1 \beta^\mu_{th}(u-s\cdot th)ds\right) = \exp\left(\int_0^t \beta^\mu_{h}(u-s'\cdot h)ds'\right)
\end{equation}
in $L^1(\mu)$.
It is rather easy to see that $\lambda_1$ (and consequently $\lambda_\ell$ for any $\ell\in\N$) is Fomin differentiable since
\begin{equation*}
	\left|\lim_{t\to 0} \frac{\lambda_1({\mathcal A}+t) -\lambda_1({\mathcal A})}t\right| =
	\lim_{t\to 0} \frac 1{2t}\left| \int_{{\mathcal A}\cap \R_-} (\e^{x+t} - \e^x)dx + \int_{{\mathcal A} \cap \R_+} (\e^{-x-t} - \e^{-x})dx \right|
	< 1
\end{equation*}
for any ${\mathcal A}\in {\mathcal B}(\R)$. In more generality, this follows from lemma 
\ref{lem:onedim_diff} since the density function is absolutely continuous.

\begin{theorem}\label{t:Besov-logder}
Let $\lambda$ be the $B^s_1$-Besov measure given in definition \ref{def:besov_prior}.
The set of differentiability is given by
$D(\lambda) = B^{s-\frac d2}_{2}(\T^d)$ and for any $h = \sum_{\ell=1}^\infty h_\ell \phi_\ell \in D(\lambda)$ we have
\begin{equation*}
	\beta_h^\lambda(u) = \lim_{N\to\infty} \sum_{\ell=1}^N \beta_{h_\ell}^{\lambda_\ell}(u_\ell)
	\quad {\rm in } \; L^1(\lambda),
\end{equation*}
where
\begin{equation*}
\beta_{h_\ell}^{\lambda_\ell}(u_\ell) =  - \ell^{\frac sd-\frac 12} \sign (u_\ell) h_\ell 
\end{equation*}
is the logarithmic derivative of $\rho_\ell$ (see notation \ref{not:densities}).
\end{theorem}

Notice that the previous theorem directly states that for any $h\in B_{1}^{s}(\T^d)$ the logarithmic derivative is bounded
$|\beta^\lambda_h(u)| \leq C \norm{h}_{B_{1}^{s}}$
$\lambda$-almost surely. 

For the posterior distribution $\mu^y$ we can solve the logarithmic derivative by using properties of the prior and the functional $\Phi$. The following result follows directly from \cite[prop. 3.3.12]{Boga10} (see also \cite[thm. 5.7]{dunlop2015map}).
\begin{theorem}
Suppose that $\Phi : B^t_1(\T^d) \to \R$, with $t<s-d$, is bounded from below and possesses a uniformly bounded derivative.
Then we have $\beta^{\mu^y}_h = -\partial_h \Phi(u) - \beta^\lambda_h(u)$ for any $h\in D(\lambda)$.
\end{theorem}

\section{Proofs}\label{sec:proofs}

\subsection{Proofs of results in section \ref{sec:background}}

\begin{proof}[Proof of proposition \ref{prop:symqcex}]
Without loss of generality assume that $\mu$ is symmetric around the origin and show that the origin is a strong mode. For any $u\in X$ the Anderson inequality (proposition \ref{c:B-anderson}) implies $\mu(B_\epsilon(u)) \leq \mu(B_\epsilon(0))$, and so $\mu(B_\epsilon(0))=\sup_{u\in X}\mu(B_\epsilon(u))$ and the origin is a strong mode of $\mu$. 
\end{proof}

\begin{proof}[Proof of proposition \ref{prop:qcunim}]
Suppose $u_1, u_2$ are strong modes. For $\kappa\in(0,1)$ we show that $\hat{u}=\kappa u_1+(1-\kappa)u_2$ is also a strong mode. For $\epsilon>0$  define  $M_\epsilon=\sup_{u\in X}\mu(B_\epsilon(u))$. By quasi-concavity and the identity 
\begin{equation}
\label{eq:sumofballs}
B_\epsilon(\kappa u_1 +(1-\kappa)u_2)=\kappa B_\epsilon(u_1) + (1-\kappa) B_\epsilon(u_2),
\end{equation}
we have that \[\frac{\mu(B_\epsilon(\hat{u}))}{M_\epsilon}\geq \frac{\min\{\mu(B_\epsilon(u_1)), \mu(B_\epsilon(u_2))\}}{M_\epsilon},\] so that since $u_1, u_2$ are strong modes we get \[\liminf_{\epsilon\to0}\frac{\mu(B_\epsilon(\hat{u}))}{M_\epsilon}\geq 1.\]
 Since for all $\epsilon>0$ we have $\mu(B_\epsilon(\hat{u}))/M_\epsilon\leq 1$,  we get that $\hat{u}$ is a strong mode.
\end{proof}

\begin{proof}[Proof of theorem \ref{thm:wmode}]
Let us consider the identity \eqref{eq:sumofballs} with values $u_1 = \hat u - h$ and $u_2 = \hat u$.
By applying log-concavity we have
\begin{equation*}
	\mu(B_\epsilon(\hat u - \kappa h)) \geq \mu(B_\epsilon(\hat u - h))^\kappa \mu(B_\epsilon(\hat u))^{1-\kappa}
\end{equation*}
and, consequently,
\begin{equation}
	\label{eq:convex_ineq}
	\frac{\mu(B_\epsilon (\hat u- h))}{\mu(B_\epsilon(\hat u))} \leq \left(\frac{\mu(B_\epsilon(\hat u - \kappa h)}{\mu(B_\epsilon(\hat u))}\right)^{1/\kappa}.
\end{equation}
for any $0\leq \kappa \leq 1$.

Now suppose $\hat u \in X$ is local mode in a neighborhood $B_\delta(\hat u)$ but not a global mode. 
Then there exists $\delta>0$ such that $\hat u$ is a local mode in the neighborhood $B_{\delta}(\hat u)$ but not in $B_{\delta +1}(\hat u)$.
That is, in the larger neighborhood $B_{\delta +1}(\hat u)$ we have for some $\eta > 0$ that there exists such a subsequence $\{\epsilon_j\}_{j=1}^\infty$ that
\begin{equation*}
	\lim_{j\to \infty} \frac{\mu(B_{\epsilon_j}(\hat u))}{M^{\epsilon_j}_{\delta+1}} < 1 - \eta.
\end{equation*}
Now let us choose a sequence $\{u_j\}_{j=1}^\infty \subset B_{\delta+1}(\hat u)$ that
\begin{equation*}
	M^{\epsilon_j}_{\delta+1} \leq \mu(B_{\epsilon_j}(u_j)) + \frac{M^{\epsilon_j}_{\delta+1}}{j}.
\end{equation*}
Then it follows that
\begin{equation*}
	\limsup_{j\to\infty} \frac{\mu(B_{\epsilon_j}(\hat u))}{\mu(B_{\epsilon_j}(u_j))} \leq \lim_{j\to \infty} \frac{\mu(B_{\epsilon_j}(\hat u))}{(1-\frac 1j) M^{\epsilon_j}_{\delta+1}} \leq 1-\eta.
\end{equation*}
Since $\hat u$ is a local mode there exists $\tilde \epsilon > 0$ such that for any $\epsilon < \tilde \epsilon$ we have
\begin{equation*}
	\frac{\mu(B_\epsilon(\hat u))}{M^\epsilon_\delta} \geq \left(1 - \frac \eta 2\right)^{\tilde \delta}
\end{equation*}
for $\tilde \delta = \frac \delta{2(1+\delta)}$ and some $\eta > 0$. In fact, for this choice of $\tilde \delta$ we have $\hat u - \tilde \delta (\hat u-u_j) \in B_\delta(\hat u)$ and it follows by \eqref{eq:convex_ineq} for any $\epsilon_j < \tilde \epsilon$ that
\begin{eqnarray*}
	1 - \frac \eta 2 & \leq & \left(\frac{\mu(B_{\epsilon_j}(\hat u))}{M^{\epsilon_j}_\delta}\right)^{1/\tilde\delta} \\
	& \leq & \left(\frac{\mu(B_{\epsilon_j}(\hat u))}{\mu(B_{\epsilon_j}(\hat u - \tilde \delta (\hat u-u_{\epsilon_j})))}\right)^{1/\tilde \delta}  \\
	& \leq & \frac{\mu(B_{\epsilon_j}(\hat u))}{\mu(B_{\epsilon_j}(u_{\epsilon_j}))} \\
	& \leq & 1 - \eta.
\end{eqnarray*}
This yields a contradiction and proves the claim for strong MAP estimates.

Suppose now that $\hat u \in X$ is a local wMAP but not a global wMAP. Assume like above that $B_\delta(\hat u)$ is the maximal neighborhood, where $\hat u$ is a local wMAP. Then there exists an element $h\in E$ and $h\notin B_\delta(\hat u)$ such that
\begin{equation*}
	\lim_{\epsilon\to 0} \frac{\mu(B_\epsilon(\hat u - \tilde \delta h))}{\mu(B_\epsilon(\hat u))} \leq 1 \quad {\rm but}
	\quad
	\lim_{\epsilon\to 0} \frac{\mu(B_\epsilon(\hat u - h))}{\mu(B_\epsilon(\hat u))} > 1 + \eta
\end{equation*}
for $\tilde \delta = \frac{\delta }{2\norm{h}_X}$, since $\hat u - \tilde \delta h \in B_\delta(\hat u)$.
Again we see that the inequality \eqref{eq:convex_ineq} yields a contradiction. This completes the proof.
\end{proof}

\subsection{Proofs of results in section \ref{sec:besov-def}}
\begin{proof}[Proof of lemma \ref{lem:logconc}]
Let $\lambda^N=\bigotimes_{\ell=1}^{N}\rho_\ell$. 
Then it is straightforward to check that $\lambda^N$ converges weakly to $\lambda$ as $N\to\infty$. 
By \cite[theorem 2.2]{CB74}, for $\lambda$ to be logarithmically concave, it suffices to show that the measures $\lambda^N$ in $\R^N$ are logarithmically concave. Note that the measures $\lambda^N$ have density, denoted $\pi_N$, with respect to the Lebesgue measure in $\R^N,$ given by 
\[\pi_N(x_1,..,x_N)=\left(\prod_{\ell=1}^N\frac{\ga_\ell}2\right)\exp\left(-\sum_{\ell=1}^N\ga_\ell|x_\ell|\right),\] 
for all $x=(x_1,..,x_N)\in\R^N$,
where $\ga_\ell=\ell^{\frac{s}d-\frac12}$ are the coefficients in the expansion defining the $B^s_1$-Besov measure. Since the density $\pi_N$ is a logarithmically concave function, by \cite[theorem 1.8.4]{Boga_gaussian} $\gl^N$ is logarithmically concave and the result follows. 
\end{proof}

\begin{proof}[Proof of lemma \ref{l:BesovRN}]
We use the Kakutani-Hellinger theory \cite[Chapter 2]{GDP06}. 
We start by calculating the Hellinger integrals $H(\rho_{h,\ell},\rho_{\ell})$, where $\rho_{h,\ell}(\cdot):=\rho_{\ell}(\cdot-h_\ell)$: 
\begin{eqnarray*}
H(\rho_{h,\ell},\rho_{\ell}) & = & \int_{-\infty}^{\infty}\sqrt{\frac{d\rho_{h,\ell}}{dx}\frac{d\rho_{\ell}}{dx}}dx \\
& =& \int_{-\infty}^{\infty}\frac{\ga_\ell}2 \e^{-\frac{\ga_{\ell}}{2}(|x-h_{\ell}|+|x|)}dx \\
& = & \e^{-\frac{\ga_\ell
}2 |h_{\ell}|}\left(1+\frac{\ga_\ell}2{|h_{\ell}|}\right).
\end{eqnarray*}
By \cite[lemma 2.5]{GDP06}, we have 
\[H(\lambda_h,\lambda)=\lim_{N\to\infty} H_N,\] 
where $H_N=\prod_{{\ell}=1}^N \e^{-\frac{\ga_\ell}2{|h_{\ell}|}}\left(1+\frac{\ga_\ell}2{|h_{\ell}|}\right)\in (0,1]$. By taking the negative logarithm we get 
\begin{multline*}
-\log\left(H(\lambda_h,\lambda)\right)=-\log\left(\lim_{N\to\infty}H_N\right) \\ =-\lim_{N\to\infty}\log(H_N)=\lim_{N\to\infty}\sum_{{\ell}=1}^N\left(\frac{\ga_\ell}2|h_{\ell}|-\log\left(1+\frac{\ga_\ell}2|h_{\ell}|\right)\right).
\end{multline*} 
By \cite[theorem 2.7]{GDP06} the set $Q(\lambda)$ coincides with the set of $h$ such that  $-\log\left(H(\lambda_h,\lambda)\right)<\infty.$\\
The Taylor theorem implies the upper and lower bounds 
\[x-\frac{x^2}2 \leq \log(1+x)\leq x-\frac{x^2}{2(1+x)^2}.\]
Using the lower bound, we get that \[-\log\left(H(\lambda_h,\lambda)\right)\leq\lim_{N\to\infty}\sum_{{\ell}=1}^N\frac12\left(\frac{\ga_\ell}2|h_{\ell}|\right)^2=\frac18\sum_{\ell=1}^\infty \ga_\ell^2|h_{\ell}|^2,\] which implies that a sufficient condition for the equivalence of $\gl_h$ and $\gl$ is \begin{equation}\label{eq:eqcond}\sum_{\ell=1}^\infty{\ell}^{\frac{2s}d-1}h_{\ell}^2<\infty.\end{equation}
Using the upper bound, and letting $x_{\ell}=\frac{\ga_\ell}2|h_{\ell}|$, we get that 
\[-\log\left(H(\lambda_h,\lambda)\right)\geq\sum_{\ell=1}^\infty\frac{x_{\ell}^2}{2(1+x_{\ell})^2}.\] 
If $x_{\ell}$ is unbounded, 
then the sum on the right hand side is infinite
and we have that $\gl_h$ and $\gl$ are singular (note that if $x_{\ell}$ is unbounded then obviously condition \eqref{eq:eqcond} does not hold). If $x_{\ell}$ is bounded, $x_{\ell}\leq M$, then \[\sum_{\ell=1}^\infty\frac{x_{\ell}^2}{2(1+x_{\ell})^2}\geq\frac{1}{2(1+M)^2}\sum_{\ell=1}^\infty x_{\ell}^2,\] therefore  condition \eqref{eq:eqcond} is also necessary for the equivalence of the measures $\gl_h$ and $\gl$.\\
Note, that condition \eqref{eq:eqcond} is equivalent to $h\in B^{s-\frac{d}2}_2(\T^d)$. 
Observing that \[\frac{d\rho_{h,\ell}}{d\rho_{\ell}}(x)=\e^{-{\ga_\ell}|x-h_{\ell}|+\ga_\ell|x|}\] 
for all $x\in \R$ and $\ell \in\N$, the claimed expression for $\frac{d\gl_h}{d\gl}$ follows from \cite[theorem 2.7]{GDP06}.\end{proof}

\begin{proof} [Proof of lemma \ref{l:crh}]We have
\begin{equation}
-|h_\ell-u_\ell|+|u_\ell|= -|h_\ell|+\big(|u_\ell|+\sign(h_\ell)u_\ell\big)
\end{equation}
if $|u_\ell|\le |h_\ell|$ and
\begin{equation}
-|h_\ell-u_\ell|+|u_\ell|=\sign(u_\ell)h_\ell
\end{equation}
if $|u_\ell|> |h_\ell|$. 
Fix $\eta>0$ and consider $v\in B^t_1(\T^d)$ such that $\|v-u\|_{B^t_1}< \eta$. 
Let us next define index sets
\begin{eqnarray*}
A_1 & = & \{\ell\in\N:|u_\ell|\le |h_\ell|,|v_\ell|\le |h_\ell|\},\\
A_2 & = & \{\ell\in\N:|u_\ell|\le |h_\ell|,|v_\ell|> |h_\ell|\},\\
A_3 & = & \{\ell\in\N:|u_\ell|> |h_\ell|,|v_\ell|\le |h_\ell|\} \quad {\rm and} \\
A_4 & = & \{\ell\in\N:|u_\ell|> |h_\ell|,|v_\ell|> |h_\ell|\}.
\end{eqnarray*}
Clearly $A_j$ are disjoint and $\N = \cup_{j=1}^4 A_j$.
We utilise the index sets $A_j$ to rewrite 
\begin{multline}
|\tr_h^\gl(v)-\tr_h^\gl(u)|
=\left|\tr_h^\gl(u)\left(\frac{\tr_h^\gl(v)}{\tr_h^\gl(u)}-1 \right)\right| \\
\le \e^{\|h\|_{B^s_1}}\left|\frac{\tr_h^\gl(v)}{\tr_h^\gl(u)}-1 \right|
\le \e^{\|h\|_{B^s_1}}\left|\exp(\cI_1+\cI_2+\cI_3+\cI_4)-1 \right|,\label{e:drh}
\end{multline}
where
\begin{align*}
\cI_j:=\sum_{\ell\in A_j} \ga_\ell\left(-|h_\ell-v_\ell|+|v_\ell|-\Big(-|h_\ell-u_\ell|+|u_\ell|\Big)\right).
\end{align*}
We continue by studying the terms $\cI_j$ separately. Let $0<\gd<1$ be a value, which we later fix.
For $\cI_1$ we have by the H\"older inequality
\begin{eqnarray*}
\left|\cI_1\right| & = & \left|\sum_{\ell\in A_1} \ga_\ell\left(-|h_\ell|+\big(|v_\ell|+\sign(h_\ell)v_\ell\big)-\Big(-|h_\ell|+\big(|u_\ell|+\sign(h_\ell)u_\ell\big)\Big)\right)\right|\\
& \le & \sum_{\ell\in A_1}\ga_\ell\big( \big||v_\ell|-|u_\ell|\big|+\left|v_\ell-u_\ell\right|    \big) \\
& \le &  2\sum_{\ell\in A_1}\ga_\ell|v_\ell-u_\ell| \\
& = & 2\sum_{\ell\in A_1} \ell^{\delta(\frac{t}{d}-\frac{1}{2})} |v_\ell-u_\ell|^{\delta} \ell^{\frac{s-\delta t}{d}-(1-\delta)\frac{1}{2}}|v_\ell-u_\ell|^{1-\delta}\\
&\leq &{\left(\sum_{\ell\in A_1}2\ell^{\frac{t}d-\frac12} |v_\ell-u_\ell|\right)^\delta \left(\sum_{\ell\in A_1}\ell^{\frac{s-\delta t}{d(1-\delta)}-\frac12}|v_\ell-u_\ell|\right)^{1-\delta}}\\
& \le & 2 \eta^\delta \norm{h}_{B^{s+\delta_1(s-t)}_1}^{1-\delta}
\end{eqnarray*}
for $\delta_1= \frac{\delta}{1-\delta}$, 
where to bound the second parenthesis in the second to last line we have used that $|v_\ell-u_\ell|\leq 2|h_\ell|$ for any $\ell\in A_1$. For $\cI_2$ we have
\begin{align*}
\cI_2
&= \sum_{\ell\in A_2} \ga_\ell\,\sign(v_\ell)h_\ell-\Big(-\ga_\ell|h_\ell|+\ga_\ell\big(|u_\ell|+\sign(h_\ell)u_\ell\big)\Big)\\
&= \sum_{I\in A_2}  \ga_\ell\Big( \sign(v_\ell)h_\ell-  \sign(h_\ell)u_\ell\Big) +\ga_\ell\big( |h_\ell|-|u_\ell| \big). 
\end{align*}
We note that since $|v_\ell|>|h_\ell|\ge |u_\ell|$, the following two inequalities hold:
\begin{eqnarray*}
	|h_\ell|-|u_\ell| & \leq & \min(|h_\ell|, |v_\ell-u_\ell|) \quad {\rm and} \\
	|{\rm sign}(v_\ell)|h_\ell| - u_\ell| & \leq & \min(2|h_\ell|, |v_\ell-u_\ell|)
\end{eqnarray*}
As a direct consequence we have
\begin{equation*}
	\max(|h_\ell|-|u_\ell|, |{\rm sign}(v_\ell)|h_\ell| - u_\ell|) \le c |h_\ell|^{1-\delta} |v_\ell-u_\ell|^\delta
\end{equation*}
for any $0\leq \delta\leq 1$.
Moreover, since
\begin{equation*}
	|\sign(v_\ell)h_\ell-  \sign(h_\ell)u_\ell| = |{\rm sign}(v_\ell)|h_\ell| - u_\ell|
\end{equation*}
we obtain
\begin{align*}
\cI_2 \le c\sum_{\ell\in A_2} \ell^{\delta(\frac{t}{d}-\frac{1}{2})}\ell^{\frac{s-\delta t}{d}-(1-\delta)\frac{1}{2}} |h_\ell|^{1-\delta}\,|v_\ell-u_\ell|^{\delta}
\,\le\, c \eta^\delta \|h\|_{B^{s + \delta_1(s-t)}_1}^{1-\delta},
\end{align*}
where as before $\delta_{1} = \frac{\delta}{1-\delta}$.
The term $\cI_3$ is very similar to $\cI_2$ (only the role of $u_\ell$ and $v_\ell$ is swapped), and we have
$$
|\cI_3| \,\le\, c\eta^\delta \,\|h\|_{B^{s+\delta_1(s-t)}_1}^{1-\delta}.
$$
{\color{black}}
For $\cI_4$ we first note that since $|v_\ell|>|h_\ell|$ and $|u_\ell|>|h_\ell|$ for $\ell\in A_4$,
if $|h_\ell|>\frac{\eta}{2\beta_\ell}$, with $\beta_\ell=\ell^{\frac{t}{d}-\frac{1}{2}},$ we must have $\sign(u_\ell) =\sign(v_\ell)$. Now it follows that
\begin{eqnarray*}
|\cI_4|\,
& = & \, \left|\sum_{\ell\in A_4} \ga_\ell\,(\sign(v_\ell)-\sign(u_\ell))h_\ell\right| \\
& \le & \, 2\sum_{A_4 \cap \{l\; ; \; |h_\ell|<\frac{\eta}{2\beta_\ell}\}} \ga_\ell\,|h_\ell|\\
&\le & {2^{1-\delta}}\eta^{\delta}\sum_{\ell=1}^\infty \ell^{\frac{s}{d}-\frac{1}{2}-\delta(\frac{t}{d}-\frac{1}{2})}|h_\ell|^{1-\delta} \\
& = & {2^{1-\delta}} \eta^{\delta} \sum_{\ell=1}^\infty \ell^{-\delta-\delta^2} \ell^{\frac{s}{d}-\frac{1}{2}-\delta(\frac{t}{d}-\frac{3}{2}-\delta)} |h_\ell|^{1-\delta} \\
& \le & {2^{1-\delta}}\eta^{\delta}\left(\sum_{\ell=1}^\infty \ell^{-1-\delta}\right)^\delta\left(\sum_{\ell=1}^\infty \ell^{(\frac{s}{d}-\frac{1}{2}-{\tilde\delta})\frac{1}{1-\delta}}|h_\ell|\right)^{1-\delta} \\
& \le & {2^{1-\delta}}c_\delta \eta^{\delta}\|h\|_{B^{s+{\delta_2}}_1}^{1-\delta},
\end{eqnarray*}
where
\begin{equation*}
	\tilde\delta = \delta\left(\frac{t}{d}-\frac{3}{2}-\delta\right) \quad {\rm and} \quad
	{\delta_2 = \frac{\delta}{1-\delta}\left(s-\frac{d}{2}\right)-\frac{d\tilde\delta}{1-\delta}}.
\end{equation*}
Combining, since for any $\gs>s$ there exists $\delta>0$ small enough so that $h\in B^{s+\delta_1(s-t)}_1(\T^d) \cap B^{s+\delta_2}_1(\T^d)$, we get that as $\eta\to0$, $\tr_h^\gl (v)-\tr_h^\gl (u)\to0$. This proves the claim.
\end{proof}

\begin{proof}[Proof of lemma \ref{l:B1CM}]
Let $A^\gep:=\gep A$. We first note that by lemma \ref{l:BesovRN} we can write
\begin{eqnarray*}
\gl_h(A^\gep) & = & \int_{A^\gep}\lim_{N\to\infty}\e^{\sum_{\ell=1}^N -\ga_\ell|h_\ell-u_\ell|+\ga_\ell|u_\ell|}\,\gl(\ud u)\\
& \ge & \int_{A^\gep}\lim_{N\to\infty}\e^{\sum_{\ell=1}^N -\alpha_\ell|h_\ell|}\,\gl(\ud u)\\
&= & \e^{-\|h\|_{B^s_1}}\int_{A^\gep}\,\gl(\ud u) \\
& = & \e^{-\|h\|_{B^s_1}} \gl(A^\gep).
\end{eqnarray*}
This implies that
\begin{align}\label{e:sball_inf}
\liminf_{\gep\to 0}\frac{\gl_h(A^\gep)}{\gl(A^\gep)}\ge  \e^{-\|h\|_{B^s_1}}.
\end{align}
Now consider $\{h^j\}_{j=1}^\infty\subset B_1^{s+1}(\T^d)$ with $h^j\to h$ in $B^s_1(\T^d)$. We have
\begin{align*}
\gl_h(A^\gep) 
&= \int_{A^\gep}\lim_{N\to\infty}\e^{\sum_{\ell=1}^N -\ga_\ell|h_\ell-u_\ell|+\ga_\ell|u_\ell|}\,\gl(\ud u)\\
&= \e^{\|h-h^j\|_{B^s_1}}\int_{A^\gep}\lim_{N\to\infty}\e^{\sum_{\ell=1}^N -\ga_\ell|h_\ell-u_\ell|+\ga_\ell|u_\ell|-\ga_\ell|h_\ell^j-h_\ell|}\,\gl(\ud u)\\
&\le \e^{\|h-h^j\|_{B^s_1}}\int_{A^\gep}\lim_{N\to\infty}\e^{\sum_{\ell=1}^N -\ga_\ell|h_\ell^j-u_\ell|+\ga_\ell|u_\ell|}\,\gl(\ud u)\\
&= \e^{\|h-h^j\|_{B^s_1}}\gl_{h^j}(A^\gep).
\end{align*}
This, using lemmas \ref{l:crh} and \ref{lem:cont_repres} implies that  
$$
\limsup_{\gep\to 0}\frac{\gl_h(A^\gep)}{\gl(A^\gep)}\le  \e^{\|h-h^j\|_{B^s_1}}\e^{-\|h^j\|_{B^s_1}},
$$
and letting $j\to\infty$ in the right-hand side gives
$$
\limsup_{\gep\to 0}\frac{\gl_h(A^\gep)}{\gl(A^\gep)}\le  \e^{-\|h\|_{B^s_1}}.
$$
The above inequality together with (\ref{e:sball_inf}) give the result.

\end{proof}

\subsection{Proofs of results in section \ref{sec:bip}}
\label{sec:proof_for_sec4}

In some of the proofs below we consider the sequences that converge in the weak*-topology of $B^t_1(\T^d)$. We note that the Banach space $B^t_1(\T^d)$ is isomorphic to a weighted $l^1$ space, and hence its pre-dual is the space of functions $\{v\in B^{-t}_\infty(\T^d): \lim_{j\to\infty}\la v,\psi_\ell\ra=0\}$.

\begin{proof}[Proof of lemma \ref{l:OMmin}]
Suppose $\{u_j\}_{j=1}^\infty \subset B^t_1(\T^d)$ is a minimizing sequence of functional $I$. Clearly, we can assume $\{I(u_j, y)\}_{j=1}^\infty$, and therefore also $\norm{u_j}_{B^s_1}$ to be bounded.
By the Banach--Alaoglu theorem there exists a
subsequence that converges to some $\hat u \in B^s_1(\T^d)$ in the weak*-topology. Notice that the norm of $B^s_1(\T^d)$ is lower semicontinuous in the weak*-topology and consequently $\hat u \in B^s_1(\T^d)$. 
We now show the strong convergence of the above subsequence in $B^\ts_1(\T^d)$ for any $\ts<s$. We have
\begin{eqnarray*}
\left\| u_j- \hat u\right\|_{B^{\ts}_1}
& = &  \sum_{\ell=1}^\infty \ell^{\frac{\ts}{d}-\frac{1}{2}}|\la u_j-\hat u,\psi_\ell\ra|\\
& \le &  \sum_{\ell=1}^N \ell^{\frac{\ts}{d}-\frac{1}{2}}|\la u_j-\hat u,\psi_\ell\ra| +N^{-\frac{s-\ts}{d}}\sum_{\ell=N+1}^\infty \ell^{\frac{s}{d}-\frac{1}{2}}|\la u_j-\hat u,\psi_\ell\ra|.
\end{eqnarray*}
Noting that $\|\hat u\|_{B^s_1}+\|u_j\|_{B^s_1}\le C < \infty$, given any $\gep>0$, $N$ can be chosen large enough, independently of $j$, such that the second term in the last line of the above inequality is bounded by $\gep/2$. Having convergence coefficient-wise, there is $M\in\N$ large enough so that for $j>M$ the first term is bounded by $\gep/2$ as well. We therefore conclude that $u_j\to \hat u$ in $B^\ts_1(\T^d)$ for any $\ts<s$ and hence in particular for $\ts=t<s-d$.
By the continuity assumption on $\Phi$, it now follows that
\begin{eqnarray*}
	\Phi(\hat u; y) + \norm{\hat u}_{B^s_1} & \leq &  \lim_{k\to\infty} \Phi(u_{j_k}; y) + \liminf_{k\to\infty} \norm{u_{j_k}}_{B^s_1} \\ & \leq & \liminf_{k\to \infty}\left(\Phi(u_{j_k}; y) + \norm{u_{j_k}}_{B^s_1}\right).
\end{eqnarray*}
Therefore, $\hat u$ must be a minimizer.
\end{proof}

\begin{proof}[Proof of theorem \ref{t:OMwMAP}]
Assume that $u_{min}\in B^s_1(\T^d)$ is a minimizer of $I(\cdot; y)$. By lemma \ref{l:crh} and Lipschitz continuity of $\G$ we know that $R_h^{\mu^y} \in C(B^t_1(\T^d))$ for any $h\in B^\gs_1(\T^d)$ with $\gs>s$. Since $u_{min} \in B^s_1(\T^d)$, we can study $R_h^{\mu^y}$ pointwise and obtain
\begin{equation*}
	\log(R^{\mu^y}_h(u_{min})) = I(u_{min})-I(u_{min}-h) \leq 0
\end{equation*}
for any $h\in B^\gs_1(\T^d)$ due to the minimizing property of $u_{min}$. Therefore, $u_{min}$ is a weak MAP.

Consider the reversed claim and assume that $\hat u$ is a weak MAP to the posterior $\mu^y$ in \eqref{e:muy}. Let us also assume that $\hat u \in B^t_1(\T^d)\setminus B^s_1(\T^d)$. Due to the continuity of $\Phi$ and lemma \ref{l:crh} we have
\begin{equation}
	\label{eq:rh_aux1}
	\log R_h^{\mu^y}(\hat u) = -\Phi(\hat u - h) + \Phi(\hat u) + \sum_{\ell=1}^\infty \alpha_\ell(-|h_\ell-u_\ell|+|u_\ell|) \leq 0
\end{equation}
for any $h\in B^\gs_1(\T^d)$, $\gs>s$. Let us construct a particular function $h^N = \sum_{\ell=1}^\infty h^N_\ell \psi_\ell \in B^\gs_1(\T^d)$ by defining its coefficient vector according to
\begin{equation*}
	(h^N_\ell)_{\ell=1}^\infty = 
	\begin{cases}
		\epsilon \hat u_\ell, & \quad l\leq N, \\
		0,	& \quad l>N,
	\end{cases}
\end{equation*}
for some small $\epsilon>0$. It follows by inequality \eqref{eq:rh_aux1} and continuity of $\Phi$ that
\begin{equation}
	\label{eq:rh_aux2}
	\epsilon \sum_{\ell=1}^N \alpha_\ell  |u_\ell| \leq |\Phi(\hat u - h^N) - \Phi(\hat u)| \leq C \epsilon,
\end{equation}
where $C>0$ is the local Lipschitz constant on the neighbourhood of $\hat u$. However, $N$ was chosen arbitrarily and by our assumption on the smoothness of $\hat u$ the sum on the left hand side of \eqref{eq:rh_aux2} does not stay bounded when $N$  increases. Therefore, inequality \eqref{eq:rh_aux2} leads to a contradiction and we must have $\hat u \in B^s_1(\T^d)$.

Assuming now that the weak MAP $\hat u \in B^s_1(\T^d)$, we can separate the sum in \eqref{eq:rh_aux1} and obtain
\begin{equation*}
	I(\hat u-h)-I(\hat u) = \log R_h^{\mu^y}(\hat u) \leq 0
\end{equation*}
for any $h\in B^\gs_1(\T^d)$. By continuity of $I$ and density of $B^\gs_1(\T^d)$ in $B^s_1(\T^d)$, we find that $\hat u$ minimizes $I$.
\end{proof}

Proof of proposition \ref{t:sMAP} relies on the following four lemmas giving some properties of the Besov prior measure we have here. We list these lemmas and their proofs first.


\begin{lemma}\label{l:bdrym}
Let $X$ be a separable  Banach space, and $B$ an open and convex  set in $\mathcal{B}(X)$. For any non-degenerate measure $\mu$ with full support we have $\mu(\partial B)=0$. 
\end{lemma}
\begin{proof}
For any $\gep>0$, there exists a cylindrical set $B_\gep$ with $B_\gep\supset B$ satisfying $\mu(B_\gep)-\mu(B)\le \gep$ \cite[lemma 2.1.6]{Boga_gaussian}.
By definition of cylindrical sets, there exists some $n\in\N$ and $B_\gep^0\in\mathcal{B}(\R^n)$ such that
$$
B_\gep=\{x\in X:(l_1(x),\dots,l_n(x))\in B^0_\gep\}.
$$
Let $h=(l_1,\dots,l_n)$. Without loss of generality we can assume that $l_j \in X^*$, $j=1,...,n$, are linearly independent and, therefore, $h$ is surjective (see discussion in \cite[Sec. 2.1.]{Boga_gaussian}). 
Now we have $B\subset h^{-1}(hB)\subset B_{\gep}$. Notice carefully that $hB$ is an open set, since $h$ is an open map by the open mapping theorem.
Hence we obtain 
\begin{align*}
\mu(\partial B)
&\le \mu(h^{-1}(hB)\setminus B)+\mu(\partial h^{-1}(hB))\\
&= \mu(h^{-1}(hB)\setminus B)+\mu(\overline{h^{-1}(hB)} \cap (h^{-1}(hB))^c) \\
&= \mu(h^{-1}(hB)\setminus B)+\mu(h^{-1}(\overline{hB}) \cap h^{-1}((hB)^c)) \\
&= \mu(h^{-1}(hB)\setminus B)+\mu(h^{-1}(\overline{hB} \cap (hB)^c)) \\
&= \mu(h^{-1}(hB)\setminus B)+\mu_n(\partial (hB)).
\end{align*}
By our assumption on non-degeneracy of $\mu$, we have that $\mu_n$ is absolutely continuous with respect to the Lebesgue measure in $\R^n$, and hence we have $\mu_n(\partial (hB))=0$. Noting that $\mu(h^{-1}(hB)\setminus B)\le \mu(B_\gep)-\mu(B)\le \gep$, the result follows.
\end{proof}

In the following $\gl$ stands for a centered $B^s_{1}$-Besov prior on $B^t_1(\T^d)$ with $t<s-d$. 
Also, we frequently consider projections of $\lambda$ to a subspace ${\rm span}\{\psi_1,\cdots,\psi_n\} \subset B^t_1(\T^d)$.
We write $P_n:B^t_1(\T^d)\to \R^n$ as 
\begin{equation}
	\label{eq:aux_projection}
	P_n u = (\langle \psi_\ell,u \rangle)_{\ell=1}^n
\end{equation}
and define $\lambda_n(A) := (\lambda \circ P_n^{-1})(P_n A)$ for any $A\in {\mathcal B}(B^t_1)$.

\begin{lemma}\label{l:limB}
Let $A\subset \mathcal{B}(B^t_1(\T^d))$ be any convex, symmetric and bounded set with diameter $\delta = \sup_{u,v \in A} \norm{u-v}_{B^t_1}>0$. For any $z\in B^t_1(\T^d)$, with $t<s-d$, we have
$$
\frac{\gl(A+z)}{\gl(A)}\le \e^{-\frac{1}{2}\|z\|_{B^t_1} + \delta}.
$$
\end{lemma}

\begin{proof}
First consider the subspace ${\rm span}\{\psi_1,\cdots,\psi_n\} \subset B^t_1(\T^d)$ and let $\tilde\ga_\ell=\ell^{\frac{t}{d}-\frac{1}{2}}$. Now recall $\ga_\ell \geq \tilde\ga_\ell$ since $s>t+d$. We have
\begin{eqnarray*}
\frac{\gl_{n}(A+z)}{\gl_{n}(A)}
& = & \frac{\int_{P_n A+P_n z}\exp\left(-\sum_{\ell=1}^n\ga_\ell|u_\ell|\right)\,\ud u}{\int_{P_nA}\exp\left(-\sum_{\ell=1}^n\ga_\ell|u_\ell|\right)\,\ud u} \\
& \le & \frac{\e^{-\frac{1}{2}\left(\|z\|_{B^t_1}-\gd\right)}\int_{P_n A+P_n z}\exp\left(-\sum_{\ell=1}^n(\ga_\ell-\frac{1}{2}\tilde\ga_\ell)|u_\ell|\right)\,\ud u}{\e^{-\frac{\gd}{2}}\int_{P_nA}\exp\left(-\sum_{\ell=1}^n(\ga_\ell-\frac{1}{2}\tilde\ga_\ell)|u_\ell|\right)\,\ud u} \\
& = & \e^{-\frac{1}{2}\|z\|_{B^t_1}+\delta}\frac{\tilde\gl_{n}(P_nA+P_nz)}{\tilde\gl_{n}(P_nA)},
\end{eqnarray*}
where $\tilde\gl_{n}=\prod_{\ell=1}^n\tilde\rho_\ell$ with $\tilde\rho_\ell\sim \tilde c_\ell\exp\left((\ga_\ell-\frac{1}{2}\tilde\ga_\ell)|u_\ell|\right)$. Since $\tilde\gl_{n}$ is logarithmically concave, by theorem 6.1 of \cite{CB74} (see the proof of proposition \ref{c:B-anderson} above) we have that $\tilde\gl_{n}(A+z)\le \tilde\gl_{n}(A)$. 

To consider the limiting case, note that $P_n^{-1}(P_n A)$ is convex and symmetric. Therefore, we can use the first part of the proof above, and write 
\begin{equation} 
\label{eq:aux_ineq_n_to_inf}
\gl(A+z) \le \gl_{n}(A+z)\le \e^{-\frac{1}{2}\|P_nz\|_{B^t_1}+\delta} \gl_n(A)\le \e^{-\frac{1}{2}\|z\|_{B^t_1}+\delta} (\gl(A)+\gep),
\end{equation}
with $\gep\to 0$ as $n\to\infty$ due to weak convergence of $\lambda_n$ to $\lambda$ and since $A$ is a continuity set by lemma \ref{l:bdrym}.
\end{proof}

\begin{lemma}\label{l:limB-nE}
Suppose that $\bar z\notin B^s_1(\T^d)$, $\{z^\delta\}_{\delta>0}\subset B^t_1(\T^d)$, $t<s-d$, and $z^\delta$ converges to $\bar z$ in the weak*-topology of $B^t_1(\T^d)$ as $\delta\to 0$. Then for any $\gep>0$ there exists $\delta$ small enough such that 
$$
\frac{\gl(A^\delta+z^\delta)}{\gl(A^\gd)}\le \gep,
$$
for $A^\gd=\gd A$ where $A$ any convex, symmetric and bounded set in $\mathcal{B}(B^t_1(\T^d))$.
\end{lemma}

\begin{proof}
Below we write $u_\ell=\la u,\psi_\ell\ra$ for any $u\in B^t_1(\T^d)$ and without losing any generality assume that $A$ has diameter $1$. Since $\bar z\notin B^s_1(\T^d)$, for any $M>0$ there is an $N$ large enough such that
$$
\sum_{\ell=1}^N \ga_\ell|\bar z_\ell|>4M.
$$
Let $\gd_0<2MN^{\frac{t-s}{d}}$. Since $z^\delta$ converges to $\bar z$ in weak*-topology as $\gd\to 0$, we have 
$\la\psi_\ell,z^\delta\ra\to \la\psi_\ell,\bar z\ra$ for all $j$, and therefore $\delta_1<\delta_0$ can be chosen small enough such that
$$
\sum_{\ell=1}^N \ga_\ell|z^\gd_\ell-\bar z_\ell|<M.
$$
Thus, for any $z\in A^{\delta_1}+z^{\delta_1}$, we can write
\begin{align*}
\sum_{\ell=1}^N \ga_\ell|z_\ell|
&\ge \sum_{\ell=1}^N \ga_\ell(|\bar z_\ell|-|z^{\gd_1}_\ell-\bar z_\ell|-|z^{\delta_1}_\ell-z_\ell|)> 4M- M-\delta_1 N^{\frac{s-t}{d}}>M.
\end{align*}
Let $\gd\le\gd_1$ be sufficiently small so that 
$$
\inf_{z\in A}\exp\left(-\frac{1}{2}\sum_{\ell=1}^N \ga_\ell|z_\ell|\right)\ge \frac{1}{2}
$$
For $\gl_n$ we have that for any $M>0$ there exist $N>0$ and $\delta_1>0$ such that for $n\geq N$ and $\delta<\delta_1$ it follows
\begin{align*}
\frac{\gl_{n}(A^\gd+{z^\gd})}{\gl_{n}(A)}
&=\frac{\int_{P_nA^\gd+P_n{z^\gd}}\exp\left(-\sum_{\ell=1}^n\ga_\ell|u_\ell|\right)\,\ud u}{\int_{P_nA^\gd}\exp\left(-\sum_{\ell=1}^n\ga_\ell|u_\ell|\right)\,\ud u}\\
&\le 2\e^{-{M}/{2}}\,\frac{\int_{P_nA^\gd+P_n{z^\gd}}\exp\left(-\sum_{\ell=1}^N\frac{\ga_\ell}{2}|u_\ell|-\sum_{\ell=N+1}^n\ga_\ell|u_\ell|\right)\,\ud u}{\int_{P_nA^\gd}\exp\left(-\sum_{\ell=1}^N\frac{\ga_\ell}{2}|u_\ell|-\sum_{\ell=N+1}^n\ga_\ell|u_\ell|\right)\,\ud u}\\
&\le 2\e^{-{M}/{2}},
\end{align*}
where in the last line we have used the fact that the density in the integrals of the third line is log concave and $A$ is absolutely convex.
Similar inequality as \eqref{eq:aux_ineq_n_to_inf} generalizes the result for $\lambda$.
\end{proof}

\begin{lemma}\label{l:limB-wc}
Suppose that $\{z^\gd\}_{\gd>0}\subset B^t_1(\T^d)$ converges in weak*-topology and not strongly in $B^t_1(\T^d)$ to $0$ as $\gd\to 0$. Then for any $\gep>0$ there exists $\gd$ small enough such that 
$$
\frac{\gl(A^\gd+z^\delta)}{\gl(A^\gd)}\le \gep
$$
for $A^\gd=\gd A$ where $A$ is any convex, bounded and symmetric set in $\mathcal{B}(B^t_1(\T^d))$.
\end{lemma}

\begin{proof}
Let $\tilde\ga_\ell=\ell^{\frac{t}{d}-\frac{1}{2}}$ and without loss of generality assume that $A$ has diameter $1$. 
For any $j\in\N$ we have $\la\psi_\ell,z^\gd\ra\to 0$, as $\gd\to 0$.
There exists a subsequence which we relabel $\{z^{\gd}\}$ for which elements there exists $\kappa>0$ such that
\begin{align}\label{e:onlyw2}
\|z^{\gd}\|_{B^t_1}>\kappa.
\end{align}
Let $M>0$ be arbitrary and choose $N$ large enough so that $\ga_\ell>M\tilde\ga_\ell$ for any $j>N$. By the weak* convergence there exists $\gd$ small enough such that
$$
\sum_{\ell=1}^N \tilde\ga_\ell|z^{\delta}_\ell|<\frac{\kappa}{4}
$$
where $z^\gd_\ell=\la\psi_\ell,z^\gd\ra$. This then, using (\ref{e:onlyw2}), means that there exists $n>N$ such that
$$
\sum_{\ell=N+1}^n \tilde\ga_\ell|z^\delta_\ell|>\frac{\kappa}{4}.
$$
Now one can show that
\begin{equation*}
	\sum_{\ell=N+1}^n \tilde\ga_\ell|x_\ell| \geq
	M\left(\frac \kappa 4 - \delta\right) \quad \textrm{for any} \; x\in A^\delta + z^\delta 
\end{equation*}
and 
\begin{equation*}
	\sum_{\ell=N+1}^n \tilde\ga_\ell|x_\ell| \leq M\delta \quad \textrm{for any} \;  x \in A^\delta.
\end{equation*}
Having these bounds we obtain
\begin{eqnarray*}
\frac{\gl_{n}(A^\gd+{z^\gd})}{\gl_{n}(A^\gd)}
& = &\frac{\int_{P_nA^\gd+P_n{z^\gd}}\exp\left(-\sum_{\ell=1}^n\ga_\ell|u_\ell|\right)\,\ud u}{\int_{P_nA^\gd}\exp\left(-\sum_{\ell=1}^n\ga_\ell|u_\ell|\right)\,\ud u}\\
&\le &\frac{\e^{-M(\frac{\kappa}{4}-\gd)}\int_{P_nA^\gd+P_n{z^\gd}}\exp\left(-\sum_{\ell=1}^N\ga_\ell|u_\ell|-\sum_{\ell=N+1}^n(\ga_\ell-M\tilde\ga_\ell)|u_\ell|\right)\,\ud u}
{\e^{-M\gd}\int_{P_nA^\gd}\exp\left(-\sum_{\ell=1}^N\ga_\ell|u_\ell|-\sum_{\ell=N+1}^n(\ga_\ell-M\tilde\ga_\ell)|u_\ell|\right)\,\ud u} \\
& \leq & e^{-M(\frac{\kappa}{4}-2\gd)},
\end{eqnarray*}
where we used the log-concavity of the integrands on the last line.
Since $\kappa$ is fixed, $M$ is arbitrary and $\delta$ decreases, the result follows for $\lambda_n$. In a same way as the previous two lemmas an inequality similar to \eqref{eq:aux_ineq_n_to_inf} yields the result.
\end{proof}


\begin{proof}[Proof of proposition \ref{t:sMAP}]

i) Without losing any generality we assume that the diameter of $A$ is $1$. The function $z \mapsto \mu^y(A^\delta + z)$ for fixed $\delta>0$ is bounded from $B^t_1(\T^d)$ to $[0,1]$. Let $\{z_j\}_{j=1}^\infty\subset B^t_1(\T^d)$ denote the maximizing sequence such that
\begin{equation}
	\label{eq:sup_by_maximizing_seq}
	\sup_{z\in X} \mu^y(A^\delta + z) = \lim_{j\to\infty} \mu^y(A^\delta + z_j).
\end{equation}
Suppose the sequence $z_j$ is unbounded. Since $\Phi\geq 0$ we have for any $z\in B^t_1(\T^d)$ that
\begin{equation*}
\mu^y(A^\delta + z)
=\frac{1}{Z}\int_{A^\delta+z}\e^{-\Phi(u)}\,\ud\gl(u)
\le \frac{1}{Z}\int_{A^\delta+z}\ud\gl(u)
\le \frac{1}{Z}\e^{-\frac{1}{2}\|z\|_X + \delta}\mu^y(A^\delta)
\end{equation*}
by lemma \ref{l:limB}. Now as $\|z_j\|_{B^t_1}\to\infty$, we have $\mu^y(A^\delta + z_j)\to 0$ which yields a contradiction. Therefore, the sequence $z_j$ must be bounded.

Next, by the Banach--Alaoglu theorem there exists a limit $w \in B^t_1(\T^d)$ in the weak-* topology. In particular, 
we have $P_n z_j \to P_n w$ in $\R^n$ for any $n>0$. 
Let us write $\mu^y_n(A) = \mu^y \circ P_n^{-1}(P_n A)$ for any $A\in {\mathcal B}(B^t_1(\T^d))$.
By weak convergence of $\mu^y_n$ to $\mu^y$ for any $\epsilon>0$ there exists $N>0$ such that for $n>N$ we have
\begin{equation*}
	\mu^y_n( A^\delta + w)= \mu^y_n(A^\delta + w) \leq \mu^y(A^\delta + w) + \epsilon,
\end{equation*}
since $A^\delta$ is a set of continuity for $\mu^y$ by lemma \ref{l:bdrym}, i.e. $\mu^y(\partial A^\delta) = 0$.
In consequence, we have
\begin{align}
\mu^y(A^\delta + z_j)  &- \mu^y(A^\delta + w) \nonumber \\
&\leq \mu^y_n(A^\delta + z_j) - \mu^y_n( A^\delta + w) + \epsilon \nonumber \\
&\longrightarrow \epsilon \quad {\rm as} \; j\to\infty
\label{eq:meas_cont:conv_to_zero}
\end{align}
where the convergence to $\epsilon$ in \eqref{eq:meas_cont:conv_to_zero} appears since $\lambda_n$ is non-degenerate. 
Since $\epsilon>0$ was arbitrary, we obtain
\begin{equation*}
	\lim_{j\to\infty}\mu^y(A^\delta + z_j)  \leq \mu^y(A^\delta + w)
\end{equation*}
and by \eqref{eq:sup_by_maximizing_seq} the supremum must be attained at $z^\delta = w \in B^t_1(\T^d)$.

ii) Since $\Phi(u)\geq 0$ and $\G$ is locally Lipschitz continuous, for any $\gd\le 1$ we have
\begin{align}\label{e:precon}
1\le \frac{\mu(B^\delta(z^\delta))}{\mu(B^\delta(0))}
\le \frac{\gl(B^\delta(z^\delta))}{\e^{-L}\gl(B^\delta(0))}
\end{align}
with a constant $L$ independent of $\gd$.
Suppose that $\{z^\gd\}$ is not bounded in $X$. Then by lemma \ref{l:limB}, for any $\gep>0$, there exists $\gd$ small enough such that
$$
\frac{\mu(B^\delta(z^\delta))}{\mu(B^\delta(0))}
\le \frac{\gl(B^\delta(z^\delta))}{\e^{-L}\gl(B^\delta(0))}<\gep.
$$
This contradicts the first inequality in (\ref{e:precon}). Hence $\{z^\delta\}$ is bounded in $B^t_1(\T^d)$ for any $t<s-d$ and 
there exists $\bar z\in B^t_1(\T^d)$ and a subsequence (also denoted by) $\{z^\delta\}\subset B^t_1(\T^d)$ such that $z^\delta$ converges to $\bar z$ in the weak*-topology.
Now lemma \ref{l:limB-nE} and \ref{l:limB-wc} together with (\ref{e:precon}) 
imply that $\bar z\in B^s_1(\T^d)$ and $z^\delta\to\bar z$ in $B^t_1(\T^d)$.

iii) We first note that by local Lipschitz continuity of $\Phi$, there exists a constant $L$ depending on $\|\bar z\|_{B^t_1}$ such that
$$
\frac{\mu^y(A^\gd+z^\gd)}{\mu^y(A^\gd+\bar z)} \le \e^{L\gd} \e^{-\Phi(z^\delta)+\Phi(\bar z)}\frac{\gl(A^\gd+z^\gd)}{\gl(A^\gd+\bar z)},
$$
and therefore, since $\Phi$ is continuous on $B^t_1(\T^d)$ and $z^\delta\to \bar z$ in $B^t_1(\T^d)$, we have
$$
\limsup_{\gd\to 0}\frac{\mu^y(A^\gd+z^\gd)}{\mu^y(A^\gd+\bar z)}\le\limsup_{\gd\to 0}\frac{\gl(A^\gd+z^\gd)}{\gl(A^\gd+\bar z)}. 
$$
Now consider a sequence $\{w^j\}_{j=1}^\infty \subset B^{s+1}_1(\T^d)$ with $w^j\to \bar z$ in $B^s_1(\T^d)$ as $j\to\infty$. Then we have 
\begin{eqnarray*}
\frac{\gl(A^\gd +z^\gd)}{\gl(A^\gd+\bar z)} & =
& \frac{\int_{A^\gd+z^\gd+w^j}\lim_{N\to\infty}\exp\left(\sum_{\ell=1}^N -\ga_\ell|w^j_\ell-u_\ell|+\ga_\ell|u_\ell|\right)\,\gl(\ud u)}
          {\exp(-\|\bar z\|_{B^s_1})\int_{A^\gd}\lim_{N\to\infty}\exp\left(\sum_{\ell=1}^N -\ga_\ell|\bar z_\ell-u_\ell|+\ga_\ell|u_\ell|+\ga_\ell|\bar z_\ell|\right)\,\gl(\ud u)}\\
&\le & \frac{\sup_{u\in A^\gd+z^\gd+w^j}\exp\left(\sum_{\ell=1}^\infty -\ga_\ell|w^j_\ell-u_\ell|+\ga_\ell|u_\ell|\right)\;}{\exp(-\|\bar z\|_{B^s_1})} \cdot \frac{\gl(A^\gd+z^\gd+w^j)}{\gl(A^\gd)}\\
& \le &\frac{\sup_{u\in A^\gd+z^\gd+w^j}\exp\left(\sum_{\ell=1}^\infty -\ga_\ell|w^j_\ell-u_\ell|+\ga_\ell|u_\ell|\right)\;}{\exp(-\|\bar z\|_{B^s_1})}.
\end{eqnarray*}
This by {lemma} \ref{l:crh} implies that
$$
\limsup_{\gd\to 0} \frac{\gl(A^\gd+z^\gd)}{\gl(A^\gd+\bar z)} \le \frac{\e^{-\|\bar z\|_{B^s_1}+\|\bar z-w^j\|_{B^s_1}}}{\e^{-\|\bar z\|_{B^s_1}}} = \e^{\|\bar z-w^j\|_{B^s_1}}
$$
and then letting $j\to\infty$ in the right-hand side we get
$$
\limsup_{\gd\to 0} \frac{\gl(A^\gd+z^\gd)}{\gl(A^\gd+\bar z)} \le 1, \;\text{hence} 
\quad
\limsup_{\gd\to 0} \frac{\mu^y(A^\gd+z^\gd)}{\mu^y(A^\gd+\bar z)} \le 1.
$$
Since by definition of $z^\gd$ we have that ${\mu^y}(A^\gd+z^\delta)\ge {\mu^y}(A^\gd+{\bar z})$, we get that
$$
\liminf_{\gd\to 0} \frac{\mu^y(A^\gd+z^\gd)}{{\mu^y}(A^\gd+\bar z)} \ge 1.
$$
It follows that 
\begin{align}\label{e:sMAPlim}
\lim_{\gd\to 0} \frac{{\mu^y}(A^\gd+z^\gd)}{{\mu^y}(A^\gd+\bar z)}=1,
\end{align}
 and therefore $\bar z$ is a strong MAP estimator.

It remains to show that $\bar z$ is a minimizer of $I$. Let $z^*:=\arg{\min}_{u\in B^s_1} I(u)$. Suppose that $\bar z$ is not a minimizer so that $I(\bar z)>I(z^*)$. We first note that by local Lipschitz continuity of $\Phi$, as before we have that
there is an $L$ depending on $\|\bar z\|_{B^t_1}$ and $\|z^*\|_{B^t_1}$ such that $$
\frac{\mu^y(A^\gd+\bar z)}{\mu^y(A^\gd+z^*)} \le \e^{L\gd}\e^{-\Phi(\bar z)+\Phi(z^*)} \frac{\gl(A^\gd+\bar z)}{\gl(A^\gd+z^*)}.
$$
Now consider a sequence $\{w^j\}_{j=1}^\infty\subset B^{s+1}_1(\T^d)$ with $w^j\to \bar z$ in $B^s_1(\T^d)$. Then, similar to what we did above, we can show that 
\begin{align*}
\lim\sup_{\gd\to 0}\frac{\mu^y(A^\gd+\bar z)}{\mu^y(A^\gd+z^*)}
\le \e^{-\Phi(\bar z)+\Phi(z^*)}\frac{\e^{-\|\bar z\|_{B^s_1}+\|\bar z-w^j\|_{B^s_1}}}{\e^{-\|{z^*}\|_{B^s_1}}}
\;\to\; \e^{-I(\bar z)+I(z^*)}
\end{align*}
as $j\to\infty$.
We now note that 
\begin{align*}
\frac{\mu^y(A^\gd+z^\delta)}{\mu^y(A^\gd+z^*)}=\frac{\mu^y(A^\gd+z^\delta)}{\mu^y(A^\gd+\bar z)}\frac{\mu^y(A^\gd+\bar z)}{\mu^y(A^\gd+z^*)}
\end{align*}
which by (\ref{e:sMAPlim}) gives
$$
\limsup_{\gd\to 0}\frac{\mu^y(A^\gd+z^\delta)}{\mu^y(A^\gd+z^*)}\le  \e^{-I(\bar z)+I(z^*)}<1.
$$
This contradicts the fact that by definition of $z^\gd$, $\mu^y(A^\gd+z^\delta)\ge \mu^y(A^\gd+z^*)$.
\end{proof}

\begin{proof}[Proof of theorem \ref{c:sMAP}]
Suppose that $\tilde z$ is a strong MAP estimator. Any strong MAP estimate is a weak MAP estimate and hence, by theorem \ref{t:OMwMAP}, $\tilde z$ is a minimizer of $I$.\\
Now let $z^*$ be a minimizer of $I$. By proposition \ref{t:sMAP} we know that there exists a strong MAP estimate $\bar z\in B^s_1(\T^d)$ which also minimises $I$. Therefore, by proposition \ref{prop:limit_is_onsager}, we have
$$
\lim_{\gep\to 0}\frac{\mu^y(A^\gep+\bar z)}{\mu^y(A^\gep+z^*)}=1.
$$
Let $z^\gep=\arg\min_{z\in X}\mu^y(A^\gep+z)$. By Definition \ref{def:map}, We can write 
$$
\lim_{\gep\to 0}\frac{\mu^y(A^\gep+z^*)}{\mu^y(A^\gep+z^\gep)}
=\lim_{\gep\to 0}\frac{\mu^y(A^\gep+\bar z)}{\mu^y(A^\gep+z^\gep)}\lim_{\gep\to 0}\frac{\mu^y(A^\gep+z^*)}{\mu^y(A^\gep+\bar z)}=1.
$$
The result follows.
\end{proof}

\begin{proof}[Proof of theorem \ref{t:consistency}]
Substituting $y_j$ by $\G(\utr)+\xi_j$, we have
\begin{align*}
\arg\min I_n(u) 
=\arg\min \Bigg\{\frac{1}{n}\|u\|_{B^s_1}+|\Sigma^{-\frac{1}{2}}(\G(\utr)-\G(u))\big|^2 
+\frac{2}{n}\sum_{j=1}^n\Big\la \Sigma^{-\frac{1}{2}}(\G(\utr)-\G(u)),\Sigma^{-\frac{1}{2}}\xi_j  \Big\ra\Bigg\}.
\end{align*}
Since $u_n$ is a minimizer of $I_n$, we can write
\begin{align*}
\frac{1}{n}\|u_n\|_{B^s_1}&+\big|\Sigma^{-\frac{1}{2}}(\G(\utr)-\G(u_n))\big|^2\\
&\le \frac{1}{n}\|\utr\|_{B^s_1}+\frac{2}{n}\big|\Sigma^{-\frac{1}{2}}(\G(\utr)-\G(u_n))\big|\,\left|\sum_{j=1}^n\Sigma^{-\frac{1}{2}}\xi_j\right|\\
&\le \frac{1}{n}\|\utr\|_{B^s_1}+\frac{1}{2}\big|\Sigma^{-\frac{1}{2}}(\G(\utr)-\G(u_n))\big|^2+\frac{2}{n^2}\Big(\sum_{j=1}^n\Sigma^{-\frac{1}{2}}\xi_j\Big) ^2,
\end{align*}
using Young's inequality. Taking the expectation and using the independence of $\{\xi_j\}$, we obtain  
\begin{align}\label{e:EGlim}
\bbE|\Sigma^{-\frac{1}{2}}(\G(\utr)-\G(u_n))\big|^2\to 0\quad \mbox{as}\quad n\to\infty.
\end{align}
and, by application of the Jensen inequality,
\begin{align}\label{e:Eu_nlim}
 \bbE \|u_n\|_{B^s_1}\le \|\utr\|_{B^s_1}+\bbE |\Sigma^{-\frac{1}{2}}\xi_1|^2 < \infty.
\end{align}

First, by (\ref{e:EGlim}), $\G(u_n)\to \G(\utr)$ in probability as $n\to \infty$. Therefore there exists a subsequence which satisfies (after labelling by $n$ again)
\begin{align}\label{e:asGlim}
\G(u_n)\to \G(\utr)\quad\mbox{almost surely}.
\end{align}
%
Now let $\{\psi_\ell\}_{\ell=1}^\infty$ be an $r$-regular orthonormal wavelet basis for $L^2(\R^d)$ with $r\ge s$. We then, by (\ref{e:Eu_nlim}), have
$$
\bbE|\la u_n,\psi_\ell \ra| \le C_\ell \bbE\|u_n\|_{B^s_1}\le C_\ell( \|\utr\|_{B^s_1}+K),\quad\mbox{for all }\ell\in\N.
$$
For $\ell=1$, the above bound implies the existence of $\{u_{n_1(k)}\}_{k\in\N}\subset \{u_n\}_{n\in\N}$ and $\eta_1\in\R$ such that
$\bbE|\la u_{n_1(k)},\psi_1\ra|\to\eta_1$ as $k\to\infty$. Considering $\ell=2,3,\dots$ successively one can similarly show the existence of
$\{u_{n_1(k)}\}_{k\in\N}\supset \{u_{n_2(k)}\}_{k\in\N}\supset\dots\{u_{n_\ell(k)}\}_{k\in\N}\supset\dots$ and $\{\eta_\ell\}\in\R^\infty$ such that 
$\bbE|\la u_{n_\ell(k)},\psi_\ell\ra|\to\eta_\ell$ for any $\ell\in\N$ as $k\to\infty$. The subsequence $\{u_{n_k(k)}\}_{k\in\N}$ hence satisfies
\begin{equation}\label{e:nkk}
\bbE|\la u_{n_k(k)},\psi_\ell\ra|\to\eta_\ell, \quad\mbox{as }k\to\infty\mbox{ for any }\ell\in\N.
\end{equation}
We relabel the above subsequence by $\{u_n\}_{n\in\N}$. Let $u^*:=\sum_{\ell=1}^\infty \eta_\ell\psi_\ell$. We have, by Jensen's inequality,
\begin{eqnarray*}
\sum_{\ell=1}^N \ell^{\frac{s}{d}-\frac{1}{2}}|\eta_\ell|
& \le & \lim_{n\to\infty}\bbE\sum_{\ell=1}^N \ell^{\frac{s}{d}-\frac{1}{2}}|\la u_{n},\psi_\ell \ra| \\
& \le & \lim_{n\to\infty}\bbE \|u_{n}\|_{B^s_1}\\
& \le & \|\utr\|_{B^s_1}+K.
\end{eqnarray*}
Hence $u^*\in B^s_1$. 
Let us now consider the strong convergence of the distribution in a larger space $B^{\ts}_1$ for $\tilde s<s$.
Take $v\in B^{-\ts}_\infty(\T^d)$ and write
\begin{eqnarray*}
\bbE \left\| u_n- u^*\right\|_{B^{\ts}_1}
& = &  \sum_{\ell=1}^\infty \ell^{\frac{\ts}{d}-\frac{1}{2}}\bbE|\la u_n-u^*,\psi_\ell\ra|\\
& \le &  \sum_{\ell=1}^N \ell^{\frac{\ts}{d}-\frac{1}{2}}\bbE|\la u_n-u^*,\psi_\ell\ra| \\
& &                                    +N^{-\frac{s-\ts}{d}}\sum_{\ell=N+1}^\infty \ell^{\frac{s}{d}-\frac{1}{2}}\bbE|\la u_n-u^*,\psi_\ell\ra|,
\end{eqnarray*}
since $\bbE\|u_n\|_{B^s_1}$ is uniformly bounded, and as 
$\sum_{\ell=1}^\infty \ell^{\frac{\ts}{d}-\frac{1}{2}}\bbE|\la  u_n-u^*,\psi_\ell\ra| < \infty$ 
we can pass the expectation inside \cite[theorem 1.38]{Rud87}.
Noting that $\|u^*\|_{B^s_1}+{\bbE}\|u_n\|_{B^s_1}\le C < \infty$, given any $\gep>0$, $N$ can be chosen large enough, independently of $n$, such that the second term in the last line of the above inequality is bounded by $\gep/2$. Having convergence coefficient-wise, there is $M\in\N$ large enough so that for $n>M$ the first term is bounded by $\gep/2$ as well. We therefore conclude that $u_n$ converges strongly to $u^*$ in $B^\ts_1(\T^d)$ in probability. 
This then implies (\cite[Lemma 4.2]{Kall02}) the existence of a subsequence, which we label $\{u_n\}_{n=1}^\infty$ again, such that 
$$
u_n\to u^*\mbox{ in }\;B^\ts_1(\T^d)\;\mbox{ almost surely and for all }\ts<s.
$$ 
This is true in particular for $t<s-d$. By continuity of $\G$ in $B^t_1(\T^d)$, we conclude that 
$$
\G(u_n)\to \G(u^*)\quad\mbox{ in }\; B^t_1(\T^d)\;\;\mbox{a.s}.
$$
This together with (\ref{e:asGlim}) give $\G(\utr)=\G(u^*)$.
\end{proof}

\begin{proof}[Proof of corollary \ref{c:consistency}]
Since $B^s_1(\T^d)$ is dense in $B^t_1(\T^d)$, for any $\gep>0$, there exists $v\in B^s_1(\T^d)$ such that $\|\utr-v\|_X\le\gep$. As $u_n$ is a minimizer of $I_n$ we have
\begin{align*}
\frac{1}{n} & \|u_n\|_{B^s_1}+\big|\Sigma^{-\frac{1}{2}}(\G(\utr)-\G(u_n))\big|^2
+\frac{2}{n}\sum_{j=1}^n\Big\la \Sigma^{-\frac{1}{2}}(\G(\utr)-\G(u_n)),\Sigma^{-\frac{1}{2}}\xi_j  \Big\ra\\
&\le \frac{1}{n}\|v\|_{B^s_1}+\big|\Sigma^{-\frac{1}{2}}(\G(\utr)-\G(v))\big|^2
+\frac{2}{n}\sum_{j=1}^n\Big\la \Sigma^{-\frac{1}{2}}(\G(\utr)-\G(v)),\Sigma^{-\frac{1}{2}}\xi_j  \Big\ra.
\end{align*}
Rearranging and using Young's inequality, similar to the previous proof, then gives
\begin{align*}
\frac{1}{2}\big|\Sigma^{-\frac{1}{2}}(\G(\utr)-\G(u_n))\big|^2
\le 2\big|\Sigma^{-\frac{1}{2}}(\G(\utr)-\G(v))\big|^2+\frac{1}{n}\|v\|_{B^s_1}+\frac{3}{n^2}\Big(\sum_{j=1}^n\Sigma^{-\frac{1}{2}}\xi_j\Big) ^2
\end{align*}
Noting that $\|v\|_{B^s_1}\le \|\utr\|_{B^s_1} + \epsilon$ and $|\Sigma^{-\frac{1}{2}}(\G(\utr)-\G(v))\big|\le C\gep$ by local Lipschitz continuity of $\G$, we have
$$
\bbE|\Sigma^{-\frac{1}{2}}(\G(\utr)-\G(u_n))\big|^2\le 
4C^2\gep^2+\frac 2n\left(\|\utr\|_{B^s_1} + \epsilon+3\bbE|\Sigma^{-\frac{1}{2}}\xi_j|^2\right).
$$ 
This implies that
$$
\limsup_{n\to\infty}\bbE|\Sigma^{-\frac{1}{2}}(\G(\utr)-\G(u_n))\big|^2
\le 4C^2\gep^2.
$$
Since $\gep$ was arbitrary, we conclude that $\lim_{n\to\infty}\bbE|\Sigma^{-\frac{1}{2}}(\G(\utr)-\G(u_n))\big|^2= 0$, and therefore, $|\Sigma^{-\frac{1}{2}}(\G(\utr)-\G(u_n))\big|\to 0$ in probability. Hence, there exists a subsequence of $\{\G(u_n)\}$ which converges to ${\G(\utr)}$ almost surely, giving the result.
\end{proof}

\subsection{Proofs of results in section \ref{sec:logder}}
\begin{proof}[Proof of theorem \ref{t:Besov-logder}]
The probability distribution $\pi_X$ on $\R$ has finite Fisher information since
\begin{equation}
	\int_\R \frac{\pi_X'(t)^2}{\pi_X(t)} dt = \int_\R \pi_X(t) dt = 1.
\end{equation}
As pointed out above we have 
$\lambda_\ell(A) = \lambda_1(\ga_\ell A)$,
for any $A \in {\mathcal B}(\R)$, where $\ga_\ell = \ell^{s/d-1/2}$. 
By proposition \ref{cor:prod_measures} we have
\begin{equation*}
	D(\lambda) = \left\{y\in\R^\infty \; | \; \sum_{\ell=1}^\infty \ga_\ell^2 y_\ell^2<\infty\right\} = \left\{y \in \R^\infty \; | \; \sum_{\ell=1}^\infty \ell^{\frac{2s}{d} - 1} y_\ell^2<\infty\right\}. \end{equation*}
and, consequently, $D(\lambda) = B^{s'}_{2}(\T^d)$ for $s'=s - \frac d2$. The rest follows from theorem \ref{thm:product_measures_conv}.
\end{proof}

\section*{Acknowledgements}

MB acknowledges support by ERC via Grant EU
FP 7 - ERC Consolidator Grant 615216 LifeInverse. 
The work by TH was supported by the Academy of Finland via project 275177.
Both MB and TH were further supported by the German Science Exchange Foundation DAAD via
Project 57162894, Bayesian Inverse Problems in Banach Space.

\bibliographystyle{abbrv}
\bibliography{references}

\end{document}